\documentclass[10pt]{article}

\usepackage{arxiv}
\usepackage{algorithm}
\usepackage{algorithmic}
\usepackage[utf8]{inputenc} % allow utf-8 input
\usepackage[T1]{fontenc}    % use 8-bit T1 fonts
\usepackage[hidelinks]{hyperref}       % hyperlinks
\usepackage{url}            % simple URL typesetting
\usepackage{booktabs}       % professional-quality tables
\usepackage{amsfonts}       % blackboard math symbols
\usepackage{nicefrac}       % compact symbols for 1/2, etc.
\usepackage{microtype}      % microtypography
\usepackage{xcolor}  
\usepackage{natbib}% colors
\usepackage{amsmath}
\usepackage{amsthm}
\usepackage{amssymb}
\usepackage{mathrsfs}
\usepackage{mathtools}
\usepackage{xcolor}

\usepackage{graphicx}
\usepackage{subcaption}
\usepackage{booktabs} % for professional tables
\usepackage{paralist} % for compactenum environment
\usepackage{enumitem}
\usepackage{tocloft}
\usepackage{setspace}
\usepackage{subcaption}
\usepackage[capitalize,noabbrev]{cleveref}
\title{When Are Nesterov Trajectories Eventually Locally Minimizing?}
\usepackage{amsthm}
\newtheorem{thm}{Theorem}[section]
\newtheorem*{thm*}{Theorem} % unnumbered
\newtheorem{lem}[thm]{Lemma}
\newtheorem*{lem*}{Lemma} 
\newtheorem{deff}[thm]{Definition}
\newtheorem*{deff*}{Definition}

\newtheorem{rmk}{Remark}

\newtheorem{prop}[thm]{Proposition}

\newtheorem{cor}[thm]{Corollary}
\newtheorem*{prob*}{Problem}

\newcommand{\R}{\mathbb{R}}

\newcommand{\Norm}[1]{\left\Vert #1 \right\Vert}

\DeclareUnicodeCharacter{210B}{\mathcal{H}}

\usepackage[textsize=tiny]{todonotes}

\newif\ifshowcomments
\showcommentstrue
\begin{document}
\author{%
  Peyman Morteza and Stephen  Wright\\
  Department of Computer Sciences\\
  University of Wisconsin--Madison\\
\texttt{\{peyman,swright\}@cs.wisc.edu}
}
\date{\mbox{}}
\maketitle

\begin{abstract}
We revisit the variational characterization of Nesterov’s celebrated method and study when Nesterov trajectories are eventually locally minimizing in a variational sense, a notion that we define precisely. We show that this question is equivalent to the nonnegativity of an associated Schr\"odinger operator acting on variations of the trajectory. For power objectives of the form
\[
f(x)=\frac{1}{p}\|x\|^p
\]
(where $\| \cdot\|$ denotes the Euclidean norm), we give a complete characterization: Among all exponents \(p\ge 2\), the Nesterov trajectory with non-zero initial position is eventually locally minimizing if and only if \(p=4\). We further interpret this exceptional exponent through the stability theory of semilinear elliptic equations.

\end{abstract}

\setcounter{tocdepth}{1}
\setlength{\cftbeforesecskip}{0.6em}
{
\begingroup
\setstretch{1.0}   % or 0.85–0.95 depending on overflow
\tableofcontents
\endgroup
}
\vspace{-1mm}
\section{Introduction}
A central problem in optimization is to minimize a convex
function $f:\R^n\to\R$,
\begin{align*}
\label{problem:optimization}
    \min_{x\in\R^n} f(x).
\end{align*}
Nesterov's accelerated gradient method \citep{nesterov1983method} is a foundational algorithm for this problem. Its fast convergence guarantees have made it influential in large-scale optimization and machine learning \citep{wright2022optimization}. Despite its simplicity, understanding the acceleration phenomenon has led to several perspectives, including continuous-time dynamics and variational characterization 
\citep{jordan2018dynamical,wibisono2016variational, su2016differential,muehlebach2019dynamical,muehlebach2021optimization}. 

\citet{su2016differential} derived the following second-order ordinary differential equation as a continuous-time analog of Nesterov's method, 
\[
\ddot X(t)+\frac{3}{t}\dot X(t)+\nabla f(X(t))=0,
\qquad t>0.
\]
This ODE provides a dynamical system framework for studying accelerated trajectories. Building on that, a variational framework, developed by \citet{wibisono2016variational} and discussed in \cite{jordan2018dynamical}, characterizes these trajectories as stationary points of the following action functional associated with a time-dependent Lagrangian, 
\begin{align}
\label{eq:defL-intro}
\begin{split}
&\mathcal{L}:\mathcal{C}_{x_a,x_b} \to \R, \quad \\
&\mathcal{L}(\gamma)
=
\int_a^b
t^3
\left(
\frac{1}{2}\Norm{\dot{\gamma}(t)}^2
-
f(\gamma(t))
\right)\,dt,
\end{split}
\end{align} 
where $\mathcal{C}_{x_a,x_b}$ is the space of sufficiently regular curves, $\gamma:[a,b]\to \R^n$, connecting $x_a$ to $x_b$. We refer to $\mathcal{L}$ as the {\it Nesterov functional}.

In this paper, we revisit the variational perspective and study when the accelerated trajectories are actually minimizers of the Nesterov functional. We begin by computing its second variation
(see Lemma \ref{lem:second-variation}). Our first key observation is that the change of variables,
\begin{equation} \label{eq:cvar}
s:=\log t, \quad u(s):=e^s h(e^s) = t h(t)
\end{equation}
transforms the second variation into an unweighted quadratic
form associated with the Schr\"odinger operator,
\[
\mathcal H_\gamma
:=
-\frac{d^2}{ds^2}
+
I-e^{2s}\nabla^2 f(\gamma(e^s)),
\]
acting on vector-valued functions on $[\log a,\log b]$ with
Dirichlet boundary conditions. In particular, we show that, nonnegativity
of the second variation is equivalent to nonnegativity of
the associated Schr\"odinger operator
(see Propositions \ref{prop:logaritmic-change-variable}, \ref{prop:schrodinger}, and \ref{prop:eventual-minimality-second-variation}). This representation implies that the Hessian scaled by the
square of time,
\[
e^{2s}\nabla^2 f(\gamma(e^s))
=
t^2\nabla^2 f(\gamma(t)),
\qquad t=e^s,
\]
is the key quantity governing the behavior of the second variation. This observation also shows that the notion of minimality is dependent on the time along the flow. In particular, the starting time and interval length play a role in determining when the action is minimal or not. Motivated by this we define the notion of {\it eventually locally minimizing} (See also  Definition \ref{deff:eventual-local-min}). 
\begin{deff*}
Let $X:[0,\infty)\to\R^n$ be a Nesterov trajectory. Fix \(T_0>0\). We say that $X$ is \emph{$T_0$-eventually locally minimizing} if for every finite $a$ and $b$ satisfying $T_0<a<b<\infty$, the restriction, $X|_{[a,b]}$, is a local minimizer of the Nesterov functional in $H^{1}$-topology, among curves satisfying, $\gamma(a)=X(a),\gamma(b)=X(b)$. We say that \(X\) is eventually locally minimizing if this holds for some \(T_0>0\). 
\end{deff*}

Intuitively speaking, this notion captures the idea that, after a sufficiently large time, the trajectory locally minimizes the action on every subsequent finite interval. Motivated by this notion, we ask the following question:

\begin{quote}
\emph{For which classes of objective functions is the Nesterov
trajectory eventually locally minimizing?}
\end{quote}
\cite{zhang2021rethinking} study the minimality of accelerated trajectories for quadratic objectives on finite intervals. Beyond introducing eventual local minimality as a natural notion for studying long-time variational behavior, we show that vanishing curvature alone is insufficient to guarantee this property. Our main result (see Theorem \ref{thm:power-characterization}) gives a complete characterization for power objective functions $f(x)=\frac{1}{p}\|x\|^p$ with $p\ge 2$. Perhaps surprisingly, we find that the only exponent for which the regular Nesterov trajectory with non-zero initial position is eventually locally minimizing is $p=4$. The proof is inspired by the stability theory of radial
solutions of semilinear elliptic equations. The connection
arises because the Nesterov ODE is the radial reduction in four dimensions of the system,
\[
\Delta U+\nabla f(U)=0,
\qquad U:\R^4\to\R^n.
\]
We provide a self-contained proof of
Theorem \ref{thm:power-characterization} inspired by techniques from
elliptic stability theory.
% and develop the PDE interpretation
% in Section \ref{sec:pde-interpretation}.

The organization of the rest of this paper is as follows. In Section \ref{sec:opt-flow-revisit}, we recall the flow formulations of gradient descent and Nesterov's accelerated method. In Section \ref{sec:var-revisit}, we review the variational characterization of the Nesterov flow and compute the first and second variations of the associated Lagrangian action. In Section \ref{sec:shrod}, we show that the local minimality of a Nesterov trajectory is governed by an associated Schr\"odinger operator, and define what we mean by an eventually locally minimizing trajectory. In Section \ref{sec:eventual-power}, we obtain a precise characterization for power objectives. In Section \ref{sec:pde-interpretation}, we give a complete PDE interpretation for our main result and in Appendix \ref{sec:proof} we present omitted proofs.

\section{Optimization Methods as Flows}
\label{sec:opt-flow-revisit}
In this section, we recall some basic facts about optimization methods. Specifically, we review the formulations of gradient descent and Nesterov's accelerated method as solutions of certain ODEs. We refer to  \citep{su2016differential} and references therein for detailed discussion. We present them here for completeness and to help set the stage for our results.
\subsection{Gradient Flow}
We first recall the basic convergence result for gradient flow, which can be viewed as the continuous-time analog of gradient descent.
For completeness we include  a proof of this result in Appendix~{\ref{sec:proof}}; see also \citep{su2016differential}.
\begin{thm}
\label{thm:cont-grad-descent}
Let \(f:\R^n\to\R\) be a convex, $C^1$ function, with a finite minimum. Let \(X:[0,\infty)\to\R^n\) be the solution of,
\[
\dot X(t)+\nabla f(X(t))=0,
\qquad
X(0)=x_0.
\]
Then, for every \(x^\star\in\operatorname{argmin} f\),
\[
f(X(t))-f(x^\star)
\le
\frac{\Norm{x_0-x^\star}^2}{2t},
\qquad t>0.
\]
In particular,
\[
f(X(t))-f(x^\star)
=
O\left(\frac{1}{t}\right).
\]
\end{thm}
\subsection{Nesterov Flow}

We now turn to the continuous-time model of Nesterov's accelerated method introduced by \citet{su2016differential}. Let $f:\R^n\to\R$ be a $C^1$, convex function, with finite minimum. The associated Nesterov ODE is,
\begin{align}
\label{eq:nesterov-ode}
\ddot X(t)+\frac{3}{t}\dot X(t)+\nabla f(X(t))=0,
\qquad t>0.
\end{align}
Because of the singular coefficient \(3/t\), the natural solution concept is the regular solution satisfying,
\[
X(0)=x_0,
\qquad
\dot X(0)=0.
\]
Given the singularity at $t=0$, one needs to be careful when discussing the above initial value problem. The following lemmas show local and global existence of solutions to the Nesterov ODE. These results are standard and we defer the proofs to Appendix \ref{sec:proof}. 
\begin{lem}[Regular solution at $t=0$]
\label{lem:reg-nesterov-sol}
Let $f:\R^n\to\R$ be $C^1$ with locally Lipschitz gradient. For every
$x_0\in\R^n$, there exists $T>0$ and a unique regular solution, $X:[0,T]\to\R^n$ of,
\begin{align*}
   & \ddot X(t)+\frac{3}{t}\dot X(t)+\nabla f(X(t))=0,
\qquad t>0,\\
& X(0)=x_0,
\qquad
\dot X(0)=0.
\end{align*}
Equivalently, $X$ is the unique continuous solution of the integral equation,
\[
X(t)
=
x_0-
\int_0^t
\tau^{-3}
\left(
\int_0^\tau
s^3\nabla f(X(s))\,ds
\right)d\tau.
\]
Moreover,
\[
X(t)
=
x_0-\frac{t^2}{8}\nabla f(x_0)+o(t^2)
\qquad
\text{as }t\to0.
\]
\end{lem}
The next lemma establishes global existence.
\begin{lem}[Global existence of regular Nesterov trajectories]
\label{lem:global-existence}
Let $f\in C^2(\mathbb R^n)$. Suppose that
$f$ is bounded below. Then, for every $x_0\in\mathbb R^n$, there exists a unique
regular solution $X:[0,\infty) \to \R^{n}$ of Nesterov ODE with $X(0)=x_0,\dot{X}(0)=0$.
\end{lem}
We now recall the basic convergence estimate for this flow. See for example \citep{su2016differential} for a proof.
\begin{thm}
\label{thm:nesterov}
Let $f:\R^n\to\R$ be a $C^1$, convex function, with finite minimum. Let \(X:[0,\infty)\to\R^n\) be the regular solution of,
\[
\ddot X(t)+\frac{3}{t}\dot X(t)+\nabla f(X(t))=0,
\qquad
X(0)=x_0,
\qquad
\dot X(0)=0.
\]
Then, for every \(x^\star\in\operatorname{argmin} f\),
\begin{equation} \label{eq:so2}
f(X(t))-f(x^\star)
\le
\frac{2\Norm{x_0-x^\star}^2}{t^2},
\qquad t>0.
\end{equation}
In particular,
\[
f(X(t))-f(x^{\star})
=
O\left(\frac{1}{t^2}\right).
\]
\end{thm}

\section{Variational Characterization of Nesterov}
\label{sec:var-revisit}

In this section, we recall the variational characterization of Nesterov's method. We review the Lagrangian action studied in \citep{wibisono2016variational} and compute its first and second variations. Other works in this direction include \cite{muehlebach2019dynamical,muehlebach2021optimization}. \cite{wibisono2016variational} showed that the Nesterov trajectory can be viewed as a critical point of this action in the following sense.  For $a,b\in\R$ with $0<a<b$, consider the space of curves connecting $x_a$ to $x_b$,
\[
\mathcal{C}_{x_a,x_b}
:=
\left\{
\gamma\in C^2([a,b];\mathbb R^n):
\gamma(a)=x_a,\;
\gamma(b)=x_b
\right\}.
\]
For each such curve, define the action functional,

\begin{align}
\label{eq:defL}
\begin{split}
&\mathcal{L}:\mathcal{C}_{x_a,x_b} \to \R, \quad \\
&\mathcal{L}(\gamma)
=
\int_a^b
t^3
\left(
\frac{1}{2}\Norm{\dot{\gamma}(t)}^2
-
f(\gamma(t))
\right)\,dt.
\end{split}
\end{align} 
It is proved in \cite{wibisono2016variational} that a curve \(\gamma\) is a critical point of \(\mathcal{L}\) if and only if it satisfies the Nesterov ODE as in (\ref{eq:nesterov-ode}). We recall this fact in the next lemma.

\begin{lem}[First variation]
Let $f\in C^2(\R^n)$, and let $\mathcal{L}$ be defined as in (\ref{eq:defL}),
where $0<a<b$. For a fixed-endpoint variation,
\[
\gamma_\varepsilon(t)=\gamma(t)+\varepsilon h(t),
\qquad
h(a)=h(b)=0,
\]
the first variation of $\mathcal{L}$ at $\gamma$ is,
\[
\delta\mathcal{L}(\gamma)[h]
=
-\int_a^b
t^3
\left\langle
\ddot{\gamma}(t)
+
\frac{3}{t}\dot{\gamma}(t)
+
\nabla f(\gamma(t)),
h(t)
\right\rangle dt.
\]
Consequently, $\gamma$ is a critical point of $\mathcal{L}$ with respect to all fixed-endpoint variations if and only if,
\[
\ddot{\gamma}(t)
+
\frac{3}{t}\dot{\gamma}(t)
+
\nabla f(\gamma(t))
=
0.
\]
\end{lem}

\begin{proof}
Consider the variation,
\[
\gamma_\varepsilon=\gamma+\varepsilon h.
\]
Differentiating the action with respect to $\varepsilon$ gives,
\[
\begin{aligned}
\left.
\frac{d}{d\varepsilon}
\mathcal{L}(\gamma_\varepsilon)
\right|_{\varepsilon=0}
&=
\int_a^b
t^3
\left(
\left\langle \dot{\gamma}(t),\dot{h}(t)\right\rangle
-
\left\langle\nabla f(\gamma(t)),h(t)\right\rangle
\right)dt.
\end{aligned}
\]
Integrating the first term by parts, and using $h(a)=h(b)=0$, we obtain,
\[
\begin{aligned}
\int_a^b
t^3
\left\langle\dot{\gamma},\dot{h}\right\rangle dt
&=
-\int_a^b
\left\langle
\frac{d}{dt}\bigl(t^3\dot{\gamma}\bigr),
h
\right\rangle dt \\
&=
-\int_a^b
t^3
\left\langle
\ddot{\gamma}
+
\frac{3}{t}\dot{\gamma},
h
\right\rangle dt.
\end{aligned}
\]
Therefore,
\[
\delta\mathcal{L}(\gamma)[h]
=
-\int_a^b
t^3
\left\langle
\ddot{\gamma}
+
\frac{3}{t}\dot{\gamma}
+
\nabla f(\gamma),
h
\right\rangle dt.
\]
Since this vanishes for every fixed-endpoint variation $h$, the fundamental lemma of the calculus of variations implies,
\[
\ddot{\gamma}
+
\frac{3}{t}\dot{\gamma}
+
\nabla f(\gamma)
=
0.
\]
Conversely, if $\gamma$ satisfies this ODE, then the first variation vanishes for every such $h$.
\end{proof}
A natural question is when the Nesterov trajectory is actually a minimizer of \(\mathcal{L}\). This question was raised and discussed, for instance, in \citep{jordan2018dynamical,zhang2021rethinking}. To study this question, one needs to compute the second variation of \(\mathcal{L}\). More specifically, we record the following lemma.
\begin{lem}[Second variation]
\label{lem:second-variation}
Let $f\in C^2(\R^n)$, and let $\mathcal{L}(\gamma)$ be defined as in \eqref{eq:defL}, 
where $0<a<b$. For a fixed-endpoint variation,
\[
\gamma_\varepsilon(t)=\gamma(t)+\varepsilon h(t),
\qquad
h(a)=h(b)=0,
\]
the second variation of $\mathcal{L}$ at $\gamma$ in the direction $h$ is,
\[
\delta^2\mathcal{L}(\gamma)[h,h]
=
\int_a^b
t^3
\left(
\Norm{\dot h(t)}^2
-
\left\langle
\nabla^2 f(\gamma(t))h(t),h(t)
\right\rangle
\right)\,dt.
\]
\end{lem}

\begin{proof}
For the variation,
\[
\gamma_\varepsilon=\gamma+\varepsilon h,
\]
we have,
\[
\mathcal{L}(\gamma_\varepsilon)
=
\int_a^b
t^3
\left(
\frac12
\Norm{\dot\gamma+\varepsilon\dot h}^2
-
f(\gamma+\varepsilon h)
\right)\,dt.
\]
Differentiating once with respect to $\varepsilon$, we obtain,
\[
\frac{d}{d\varepsilon}
\mathcal{L}(\gamma_\varepsilon)
=
\int_a^b
t^3
\left(
\left\langle
\dot\gamma+\varepsilon\dot h,\dot h
\right\rangle
-
\left\langle
\nabla f(\gamma+\varepsilon h),h
\right\rangle
\right)\,dt.
\]
Differentiating once more gives,
\[
\frac{d^2}{d\varepsilon^2}
\mathcal{L}(\gamma_\varepsilon)
=
\int_a^b
t^3
\left(
\Norm{\dot h}^2
-
\left\langle
\nabla^2 f(\gamma+\varepsilon h)h,h
\right\rangle
\right)\,dt.
\]
Evaluating at $\varepsilon=0$, we obtain,
\[
\delta^2\mathcal{L}(\gamma)[h,h]
=
\int_a^b
t^3
\left(
\Norm{\dot h}^2
-
\left\langle
\nabla^2 f(\gamma)h,h
\right\rangle
\right)\,dt.
\]
\end{proof}
Intuitively speaking, when the second variation along a Nesterov trajectory is positive (more precisely, coercive), all sufficiently nearby, distinct curves with the same endpoints have larger values of the Nesterov functional. We now recall the definition of the required function space so that these notions can be studied precisely.
\begin{deff}[$H^1$ and $H_0^1$ spaces]
\label{deff:sobolev-space}
Let $0<a<b<\infty$. We define,
\[
H^1((a,b);\R^n)
:=
\left\{
h:(a,b)\to\R^n:
h,\dot h\in L^2((a,b);\R^n)
\right\},
\]
equipped with the norm,
\[
\|h\|_{H^1(a,b)}^2
:=
\int_a^b
\left(
\|h(t)\|^2
+
\|\dot h(t)\|^2
\right)\,dt.
\]

We also define,
\[
H_0^1((a,b);\R^n)
:=
\left\{
h\in H^1((a,b);\R^n):
h(a)=h(b)=0
\right\}.
\]
Here $\dot h$ denotes the weak derivative. Endpoint values
are understood using the continuous representative of
$h\in H^1((a,b);\R^n)$ on $[a,b]$.
\end{deff}
\begin{rmk}
Note that $\mathcal L$ naturally extends to
$H^1((a,b);\mathbb R^n)$ with the prescribed
endpoints because
$H^1((a,b);\mathbb R^n)$ embeds into
$C^0([a,b];\mathbb R^n)$ with an absolutely continuous representative. We sometimes use this extension.
\end{rmk}
\begin{deff}[Local minimizer on a finite interval]
Let $0<a<b<\infty$, and let,
\[
\mathcal{C}_{x_a,x_b}
:=
\left\{
\gamma\in C^2([a,b];\mathbb R^n):
\gamma(a)=x_a,\;
\gamma(b)=x_b
\right\}.
\]
A curve $\gamma\in\mathcal{C}_{x_a,x_b}$ is called a \emph{local minimizer of $\mathcal{L}$ on $[a,b]$} if there exists $\varepsilon>0$ such that
\[
\mathcal{L}(\gamma)
\le
\mathcal{L}(\eta),
\]
for every $\eta\in\mathcal{C}_{x_a,x_b}$ satisfying
\[
\|\eta-\gamma\|_{H^1}<\varepsilon.
\]
It is called a \emph{strict local minimizer of $\mathcal{L}$ on $[a,b]$} if the inequality is strict whenever $\eta\neq\gamma$.
\end{deff}
We state the following lemma. The proof is standard and we defer it to Appendix \ref{sec:proof}.
\begin{lem}[Coercive second variation implies local minimality]
\label{lem:Coercive-second-variation-local-minimality}
Let $f\in C^2(\R^n)$, let $0<a<b<\infty$, and let
$\gamma:[a,b]\to\R^n$ be a critical point of,
\[
\mathcal{L}(\eta)
=
\int_a^b
t^3
\left(
\frac12\Norm{\dot\eta(t)}^2-f(\eta(t))
\right)\,dt,
\]
with respect to fixed-endpoint variations. Suppose that the second
variation is coercive on $H_0^1((a,b);\R^n)$; that is, there exists
$c>0$ such that,
\[
\delta^2\mathcal{L}(\gamma)[h,h]
\geq
c\Norm{h}_{H^1(a,b)}^2,
\]
for every $h\in H_0^1((a,b);\R^n)$. Then $\gamma$ is a strict $H^1$-local minimizer of $\mathcal{L}$ on $[a,b]$.
More precisely, there exists $\varepsilon>0$ such that,
$\mathcal{L}(\gamma+h)>\mathcal{L}(\gamma)$, for every nonzero $h\in H_0^1((a,b);\R^n)$ satisfying, $\Norm{h}_{H^1(a,b)}<\varepsilon$.
\end{lem}
\begin{rmk}
Note that on a finite interval $[a,b]$, coercivity of the
second variation is equivalent to strict positivity
(see Lemma~\ref{lem:strict-positivity-implies-coercivity}),
so Lemma~\ref{lem:Coercive-second-variation-local-minimality}
also holds when the second variation is only assumed to be
strictly positive.
\end{rmk}

\section{Schr\"odinger Operators and Eventual Local Minimality}
\label{sec:shrod}
We need an unweighted version of the second-variation formula to better understand its positivity. The following proposition gives another characterization of the second variation of the Nesterov functional. The key here is a change of variables that leads to an unweighted representation. This has several consequences. First, it links coercivity of the second variation to positivity of a certain Schr\"odinger operator (see Proposition \ref{prop:schrodinger}). Second, it reveals that $t^{2}\nabla^{2}f(X(t))$, rather than $\nabla^{2}f(X(t))$, is the dimensionless quantity along the flow that plays a key role in determining whether the trajectory is a saddle point or a local minimizer. This result makes use of the change of variables \eqref{eq:cvar}.
\begin{prop}
\label{prop:logaritmic-change-variable}
Let \(\gamma:[a,b]\to\R^n\), with \(0<a<b<\infty\), be a critical point of the Nesterov functional,
\[
\mathcal{L}(\gamma)
=
\int_a^b
t^3
\left(
\frac12\Norm{\dot \gamma(t)}^2-f(\gamma(t))
\right)\,dt.
\]
For \(h\in H_0^1((a,b);\R^n)\), from Lemma \ref{lem:second-variation}, the second variation is,
\[
Q_\gamma[h]
=
\int_a^b
t^3
\left(
\Norm{\dot h(t)}^2
-
\left\langle
\nabla^2 f(\gamma(t))h(t),h(t)
\right\rangle
\right)\,dt.
\]
For logarithmic time $s=\log t$ and define \(u\) by $u(s)=th(t)$, as in \eqref{eq:cvar}. 
Then defining the matrix curve $K(s):=e^{2s}\nabla^2f(\gamma(e^s))$, we have,
\[
Q_\gamma[h]
=
\int_{\log a}^{\log b}
\left[
\Norm{\dot u(s)}^2
+
\left\langle
(I-K(s))u(s),u(s)
\right\rangle
\right]\,ds.
\]
\end{prop}

\begin{proof}
Since $h(t)=t^{-1}u(s)$, we have,
\[
\dot h(t)
=
-\frac{1}{t^2}u(s)
+
\frac{1}{t}\dot u(s)\frac{ds}{dt} 
=
\frac{1}{t^2}\bigl(\dot u(s)-u(s)\bigr).
\]
Therefore, we have,
\[
\begin{aligned}
\int_a^b t^3\Norm{\dot h(t)}^2\,dt
&=
\int_{\log a}^{\log b}
\Norm{\dot u(s)-u(s)}^2\,ds.
\end{aligned}
\]
Similarly,
\[
\int_a^b
t^3
\left\langle
\nabla^2f(\gamma(t))h(t),h(t)
\right\rangle dt 
=
\int_{\log a}^{\log b}
e^{2s}
\left\langle
\nabla^2f(\gamma(e^s))u(s),u(s)
\right\rangle ds,
\]
so we have that,
\[
Q_\gamma[h]
=
\int_{\log a}^{\log b}
\left[
\Norm{\dot u-u}^2
-
e^{2s}
\left\langle
\nabla^2f(\gamma(e^s))u,u
\right\rangle
\right]ds.
\]
For  the first term in the integrand, we have,
\[
\Norm{\dot u-u}^2
=
\Norm{\dot u}^2+\Norm{u}^2
-
2\langle\dot u,u\rangle.
\]
Since,
\[
2\langle\dot u,u\rangle
=
\frac{d}{ds}\Norm{u}^2,
\]
and \(u(\log a)=u(\log b)=0\), the cross term integrates to zero, that is,
\[
\int_{\log a}^{\log b}
2\langle\dot u,u\rangle\,ds
=
\left[\Norm{u(s)}^2\right]_{\log a}^{\log b}
=
0.
\]
It follows that,
\[
Q_\gamma[h]
=
\int_{\log a}^{\log b}
\left[
\Norm{\dot u(s)}^2
+
\left\langle
\left(
I-e^{2s}\nabla^2 f(\gamma(e^s))
\right)u(s),u(s)
\right\rangle
\right]ds.
\]
as claimed.
\end{proof}
As an application of Proposition \ref{prop:logaritmic-change-variable}, the next proposition relates positivity of the second variation to positivity of the associated Schr\"odinger operator acting on the corresponding space of variations.
\begin{prop}
\label{prop:schrodinger}
Let $f\in C^2(\R^n)$, let $0<a<b<\infty$, and let
$\gamma\in C^2([a,b];\R^n)$ be a critical point of the
Nesterov action with respect to fixed-endpoint variations.
Define the associated Schr\"odinger operator,
\[
\mathcal H_\gamma
:=
-\frac{d^2}{ds^2}
+
I-e^{2s}\nabla^2 f(\gamma(e^s)),
\]
on $[\log a,\log b]$ with Dirichlet boundary conditions. Then the second variation of the Nesterov action is coercive
on $H_0^1((a,b);\R^n)$ if and only if,
$\lambda_{\text{min}}(\mathcal H_\gamma)>0$.
\end{prop}
\begin{proof}
Set $V_\gamma(s)
:=
I-e^{2s}\nabla^2f(\gamma(e^s))$. For $h\in H_0^1((a,b);\R^n)$, define,
\[
u(s):=e^s h(e^s).
\]
This change of variables maps $H_0^1((a,b);\R^n)$ bijectively
onto $H_0^1((\log a,\log b);\R^n)$. In particular,
\[
u(\log a)=a h(a)=0,
\qquad
u(\log b)=b h(b)=0.
\]
By Proposition \ref{prop:logaritmic-change-variable},
\[
Q_\gamma[h]=q_\gamma[u]:
=\int_{\log a}^{\log b}
\left(
\Norm{u_s}^2
+
\left\langle V_\gamma(s)u,u\right\rangle
\right)\,ds.
\]

For $u\in H^2((\log a,\log b);\R^n)\cap H_0^1((\log a,\log b);\R^n)$,
integration by parts gives,
\[
\begin{aligned}
\left\langle u,-u_{ss}\right\rangle_{L^2}
&=
-\left[\left\langle u,u_s\right\rangle\right]_{\log a}^{\log b}
+
\int_{\log a}^{\log b}\Norm{u_s}^2\,ds\\
&=
\int_{\log a}^{\log b}\Norm{u_s}^2\,ds.
\end{aligned}
\]
Thus, for such $u$,
\[
Q_\gamma[h]=
\left\langle
u,\mathcal H_\gamma u
\right\rangle_{L^2}.
\]
For general $u\in H_0^1((\log a,\log b);\R^n)$, \(q_\gamma[u]\) is the quadratic form associated with \(\mathcal{H}_\gamma\). Suppose first that
$\lambda:=\lambda_{\text{min}}(\mathcal H_\gamma)>0$.
The Rayleigh quotient characterization gives,
\[
Q_\gamma[h]
=
q_\gamma[u]
\ge
\lambda\Norm{u}_{L^2(\log a,\log b)}^2.
\]
Next, we show how this $L^2$ lower bound yields coercivity in the $H^1$ norm. Since $V_\gamma$ is continuous on $[\log a,\log b]$, the constant,
\[
M:=\sup_{s\in[\log a,\log b]}\Norm{V_\gamma(s)}_{\mathrm{op}}
\]
is finite. Consequently,
\[
\begin{aligned}
\Norm{u_s}_{L^2(\log a,\log b)}^2
&=
q_\gamma[u]
-
\int_{\log a}^{\log b}\left\langle V_\gamma(s)u,u\right\rangle\,ds\\
&\le
q_\gamma[u]+M\Norm{u}_{L^2(\log a,\log b)}^2\\
&\le
\left(1+\frac{M}{\lambda}\right)q_\gamma[u].
\end{aligned}
\]
Adding the bound
$\Norm{u}_{L^2(\log a,\log b)}^2\le\lambda^{-1}q_\gamma[u]$ gives,
\[
\Norm{u}_{H^1(\log a,\log b)}^2
\le
\left(1+\frac{M+1}{\lambda}\right)q_\gamma[u].
\]
To compare this with the norm of $h$, observe that,
\[
h(t)=t^{-1}u(\log t),
\qquad
\dot h(t)=t^{-2}\bigl(u_s(\log t)-u(\log t)\bigr).
\]
Using $t=e^s$ and $dt=e^s\,ds$, we obtain,
\[
\begin{aligned}
\Norm{h}_{H^1(a,b)}^2
&=
\int_{\log a}^{\log b} e^{-s}\Norm{u}^2\,ds
+
\int_{\log a}^{\log b} e^{-3s}\Norm{u_s-u}^2\,ds\\
&\le
a^{-1}\Norm{u}_{L^2(\log a,\log b)}^2
+
2a^{-3}
\left(
\Norm{u_s}_{L^2(\log a,\log b)}^2
+
\Norm{u}_{L^2(\log a,\log b)}^2
\right)\\
&\le
\left(a^{-1}+2a^{-3}\right)
\Norm{u}_{H^1(\log a,\log b)}^2.
\end{aligned}
\]
Combining these estimates yields,
\[
Q_\gamma[h]
\ge
c\Norm{h}_{H^1(a,b)}^2,
\]
where,
\[
c
:=
\left[
\left(a^{-1}+2a^{-3}\right)
\left(1+\frac{M+1}{\lambda}\right)
\right]^{-1}
>0.
\]
This proves coercivity. Conversely, suppose that,
\[
Q_\gamma[h]
\ge
c\Norm{h}_{H^1(a,b)}^2
\qquad
\text{for every }h\in H_0^1((a,b);\R^n)
\]
for some $c>0$. For any $u\in H_0^1((\log a,\log b);\R^n)$, let
$h(t):=t^{-1}u(\log t)$. Then,
\[
\Norm{u}_{L^2(\log a,\log b)}^2
=
\int_a^b t\Norm{h(t)}^2\,dt
\le
b\Norm{h}_{H^1(a,b)}^2.
\]
Therefore,
\[
q_\gamma[u]
=
Q_\gamma[h]
\ge
\frac{c}{b}\Norm{u}_{L^2(\log a,\log b)}^2.
\]
Taking the infimum of the Rayleigh quotient over nonzero
$u\in H_0^1((\log a,\log b);\R^n)$ gives,
\[
\lambda_{\text{min}}(\mathcal H_\gamma)
\ge
\frac{c}{b}
>0.
\]
This proves the equivalence. Whenever these equivalent conditions hold, Lemma \ref{lem:Coercive-second-variation-local-minimality}
implies that $\gamma$ is a strict
$H^1$-local minimizer of the action.
\end{proof}
Another important aspect of Proposition \ref{prop:logaritmic-change-variable} is that it highlights $t^{2}\nabla^{2}f(\gamma(t))$ as a key quantity governing the local minimality or saddle behavior of the trajectory. In particular, it shows that both time and the Hessian play a role. For instance, the minimality properties of a trajectory at large times may differ from those near the initial time. Motivated by this observation, we introduce the following definition, which captures trajectories that, after a finite time, locally minimize the Nesterov functional on every subsequent finite interval.
\begin{deff}[Eventually locally minimizing]
\label{deff:eventual-local-min}
Let $X:[0,\infty)\to\R^n$ be a Nesterov trajectory. Fix \(T_0>0\). We say that $X$ is \emph{$T_0$-eventually locally minimizing} if for every finite $a$ and $b$ satisfying $T_0<a<b<\infty$, the restriction, $X|_{[a,b]}$, is a local minimizer of the action,
\[
\mathcal{L}_{a,b}(\gamma)
=
\int_a^b
t^3
\left(
\frac12\Norm{\dot{\gamma}(t)}^2
-
f(\gamma(t))
\right)\,dt,
\]
among curves in $H^{1}$-topology satisfying,
\[
\gamma(a)=X(a),
\qquad
\gamma(b)=X(b).
\]
We say that \(X\) is eventually locally minimizing if this holds for some \(T_0>0\). If every such restriction is a strict local minimizer, we say that $X$ is \emph{eventually strictly locally minimizing}.
\end{deff}
In light of Proposition \ref{prop:schrodinger}, we know that, on a finite interval, strict positivity of the Schr\"odinger operator is equivalent to strict positivity of the second variation. However, when studying eventual local minimality, we can actually show that more is true. In fact, nonnegativity of the Schr\"odinger operator on every finite interval after some time is sufficient to conclude eventual local minimality. The idea is very simple. If there is a nonzero eigenfunction corresponding to a zero eigenvalue on some interval, extend it by zero to a slightly larger interval that is still after that time. Nonnegativity on the larger interval implies that the extended function is a weak solution of the Schr\"odinger equation there. Since it vanishes on an open interval, uniqueness for the ODE implies that it must vanish everywhere, giving a contradiction. The following proposition makes this precise.
\begin{prop}
\label{prop:eventual-minimality-second-variation}
Let $f\in C^2(\mathbb R^n)$, let $X$ be a Nesterov
trajectory, and fix $T_0>0$. Then the following conditions are equivalent:
\begin{enumerate}
\item $X$ is $T_0$-eventually locally minimizing.
\item For every $T_0<a<b<\infty$,
\[
Q_{a,b}[h]\ge 0
\qquad
\text{for every }
h\in H_0^1((a,b);\mathbb R^n),
\]
\end{enumerate}
where $Q_{a,b}$ is the second-variation quadratic form on $[a,b]$. Moreover, whenever these conditions hold, the second
variation is coercive on every such finite interval. Consequently, every restriction $X|_{[a,b]}$ is a strict
$H^1$-local minimizer on $[a,b]$ of the Nesterov action.
\end{prop}

\begin{proof}
Suppose first that $X$ is $T_0$-eventually locally
minimizing. Fix $T_0<a<b<\infty$. For every smooth
fixed-endpoint variation $h$, the second-order necessary
condition for local minimality gives, $Q_{a,b}[h]\ge 0$.  Conversely, suppose that the second condition holds.
Fix $T_0<a<b<\infty$, and set,
\[
\alpha=\log a,
\qquad
\beta=\log b,
\qquad
V(s)=I-e^{2s}\nabla^2 f(X(e^s)).
\]
Under the change of variables $u(s)=e^s h(e^s)$, Proposition~4.1 gives,
\[
Q_{a,b}[h]
=
q_{\alpha,\beta}[u]
:=
\int_\alpha^\beta
\left(
\|u_s\|^2+\langle V(s)u,u\rangle
\right)\,ds.
\]
Thus the Dirichlet Schr\"odinger operator,
\[
\mathcal{H}_{\alpha,\beta}
=
-\frac{d^2}{ds^2}+V(s),
\]
has nonnegative lowest eigenvalue. We claim that its
lowest eigenvalue is in fact strictly positive. Suppose otherwise. There is then a nonzero eigenfunction, $u\in
H^2((\alpha,\beta);\mathbb R^n)
\cap H_0^1((\alpha,\beta);\mathbb R^n)$ such that,
\[
\mathcal{H}_{\alpha,\beta}u=0,
\qquad
q_{\alpha,\beta}[u]=0.
\]
Choose a slightly larger interval satisfying,
\[
\log T_0<\alpha'<\alpha<\beta<\beta'<\infty,
\]
and extend $u$ by zero to $(\alpha',\beta')$.
Denote this extension by $\widetilde u$. Since $u$
vanishes at both endpoints,
\[
\widetilde u\in
H_0^1((\alpha',\beta');\mathbb R^n),
\qquad
q_{\alpha',\beta'}[\widetilde u]=0.
\]

By assumption, $q_{\alpha',\beta'}$ is nonnegative.
For any
$v\in H_0^1((\alpha',\beta');\mathbb R^n)$
and any $\varepsilon\in\mathbb R$, we therefore have,
\[
0\le
q_{\alpha',\beta'}[\widetilde u+\varepsilon v]
=
2\varepsilon B(\widetilde u,v)
+
\varepsilon^2q_{\alpha',\beta'}[v],
\]
where,
\[
B(w,v)
:=
\int_{\alpha'}^{\beta'}
\left(
\langle w_s,v_s\rangle
+
\langle V(s)w,v\rangle
\right)\,ds.
\]
Considering both signs of sufficiently small
$\varepsilon$ yields,
\[
B(\widetilde u,v)=0
\qquad
\text{for every }v\in H_0^1((\alpha',\beta');\mathbb R^n).
\]
Hence $\widetilde u$ is a weak solution of,
\[
-\widetilde u_{ss}+V(s)\widetilde u=0
\qquad
\text{on }(\alpha',\beta').
\]
Because $V$ is continuous, this weak solution is $C^2$.
But $\widetilde u$ vanishes on the open interval
$(\alpha',\alpha)$. Uniqueness of the ODE therefore implies
$\widetilde u\equiv0$, contradicting the choice of $u$. We conclude that,
\[
\lambda_{\min}(H_{\alpha,\beta})>0.
\]
Proposition~\ref{prop:schrodinger} now gives coercivity of $Q_{a,b}$.
Since $X$ is a critical point of the action,
Lemma~\ref{lem:Coercive-second-variation-local-minimality} implies that $X|_{[a,b]}$ is a strict
$H^1$-local minimizer on $[a,b]$. The interval $[a,b]$ was arbitrary,
which completes the proof.
\end{proof}
The next corollary provides a sufficient condition for $X$ to be eventually  locally minimizing.
\begin{cor}
\label{cor:time-square-hessian}
Let $X:[0,\infty)\to\R^n$ be a regular Nesterov trajectory. Suppose that,
\[
\limsup_{t\to\infty}
\lambda_{\max}
\left(
t^2\nabla^2 f(X(t))
\right)
<1.
\]
Then $X$ is eventually strictly locally minimizing.

\end{cor}

\begin{proof}
By assumption, there exist $\delta>0$ and $T_0>0$ such that,
\[
\lambda_{\max}
\left(
t^2\nabla^2 f(X(t))
\right)
\leq
1-\delta,
\]
for every $t\geq T_0$. Equivalently,
\[
t^2\nabla^2 f(X(t))
\preceq
(1-\delta)I,
\qquad t\geq T_0.
\]

Fix any finite interval $T_0<a<b<\infty$. Using Proposition \ref{prop:logaritmic-change-variable} the second variation is,
\[
Q_X[h]
=
\int_{\log a}^{\log b}
\left[
\Norm{\dot u(s)}^2
+
\left\langle
\left(
I-e^{2s}\nabla^2 f(X(e^s))
\right)u(s),u(s)
\right\rangle
\right]ds.
\]
Since $e^s\geq T_0$ on this interval,
\[
I-e^{2s}\nabla^2 f(X(e^s))
\succeq
\delta I.
\]
Therefore,
\[
Q_X[h]
\geq
\int_{\log a}^{\log b}
\left(
\Norm{\dot u(s)}^2
+
\delta\Norm{u(s)}^2
\right)ds.
\]
In particular, the associated Schr\"odinger operator has strictly
positive lowest Dirichlet eigenvalue on every such interval. 
\end{proof}

\section{Eventual Local Minimality for Power Objectives}
\label{sec:eventual-power}
In this section, we prove our main result concerning eventual local minimality of the Nesterov trajectory for functions of the form  $f(x) = \frac{1}{p} \|x\|^p$, where $\|\cdot \|$ denotes the Euclidean norm and $p \ge 2$. 
The minimizer of this function is $x^*=0$. The self-contained proof of the following result relies entirely on ODE techniques and is inspired by earlier work on radial solutions of semilinear elliptic equations and their oscillatory behavior (see \cite{kawano1990generalized,kusano1987oscillation}). 
We discuss the connection with this class of PDEs in the next section.
\begin{thm}
\label{thm:power-characterization}
Let
$f(x)=\frac{1}{p}\Norm{x}^{p}$ for some $p\ge2$, 
and let $X:[0,\infty)\to\R^n$ be the regular Nesterov trajectory,
\[
\ddot X(t)+\frac{3}{t}\dot X(t)
+\Norm{X(t)}^{p-2}X(t)=0,
\qquad
X(0)=x_0\neq0,
\qquad
\dot X(0)=0.
\]
Then $X$ is eventually locally minimizing if and only if $p=4$.
\end{thm}

\begin{proof}
Since the objective is radial and $\dot X(0)=0$, uniqueness of the
Nesterov flow implies that the trajectory remains on the line spanned by
$x_0$. Thus, writing,
\[
\mathbf{e}=\frac{x_0}{\Norm{x_0}},
\qquad
X(t)=r(t)\mathbf{e},
\]
the scalar function $r$ satisfies,
\begin{equation} \label{eq:se1}
\ddot r(t)+\frac{3}{t}\dot r(t)
+|r(t)|^{p-2}r(t)=0,
\qquad
r(0)=\Norm{x_0},
\qquad
\dot r(0)=0.
\end{equation}
Moreover, we have,
\[
\nabla^2 f(X(t))
=
|r(t)|^{p-2}
\left(
I+(p-2)\mathbf{e}\mathbf{e}^\top
\right).
\]
\paragraph{Case $p=2$.}
In this case,
\[
f(x)=\frac12\Norm{x}^2,
\qquad
\nabla^2f(x)=I.
\]
Applying the change of variables \eqref{eq:cvar}, the second variation takes the form,
\[
Q_X[h]
=
\int_{\log a}^{\log b}
\left[
\Norm{\dot u(s)}^2
+
\left(1-e^{2s}\right)\Norm{u(s)}^2
\right]\,ds.
\]
Hence the second variation is independent of the particular Nesterov
trajectory $X(\cdot)$. 
Fix any $\ell>0$ and any unit vector $v\in\R^n$. For $S>0$, consider
the interval $[S,S+\ell]$ in logarithmic time and define,
\[
u(s)
=
\sin\left(
\frac{\pi(s-S)}{\ell}
\right)v,
\qquad
s\in[S,S+\ell].
\]
Then $u(S)=u(S+\ell)=0$, and,
\[
\int_S^{S+\ell}\Norm{\dot u(s)}^2\,ds
=
\frac{\pi^2}{\ell^2}
\int_S^{S+\ell}\Norm{u(s)}^2\,ds.
\]
Moreover, since $s\geq S$, $1-e^{2s}
\leq
1-e^{2S}$. It follows that,
\[
\begin{aligned}
Q_X[h]
&\leq
\left(
\frac{\pi^2}{\ell^2}
+
1-e^{2S}
\right)
\int_S^{S+\ell}\Norm{u(s)}^2\,ds.
\end{aligned}
\]
For all sufficiently large $S$, $e^{2S}
>
1+\frac{\pi^2}{\ell^2}$, and hence, $Q_X[h]<0$. The corresponding variation in the original time variable is,
\[
h(t)
=
\frac1t
\sin\left(
\frac{\pi(\log t-S)}{\ell}
\right)v,
\qquad
t\in[e^S,e^{S+\ell}],
\]
and it vanishes at both endpoints. 
Taking $S$ arbitrarily large shows that negative second-variation directions exist on finite intervals after any given time. 
Therefore, the Nesterov trajectory is not eventually locally minimizing when $p=2$.

\paragraph{Case $2<p<4$.}
In this case, it suffices to show that $r(t)$ has a sequence of zeros tending to infinity, meaning that the trajectory oscillates around the minimizer. Once this is established, the remainder of the proof is straightforward as we explain now. Let $a<b$ be two consecutive zeros of $r$. The variation,
\[
h(t)=X(t)=r(t)\mathbf{e},
\]
satisfies,
\[
h(a)=h(b)=0.
\]
Writing the scalar equation \eqref{eq:se1} in divergence form,
\begin{equation} \label{eq:pd1}
\bigl(t^3\dot r(t)\bigr)'
+
t^3|r(t)|^{p-2}r(t)
=
0,
\end{equation}
multiplying by $r(t)$, and integrating over $[a,b]$, we obtain,
\[
\int_a^b
r(t)\bigl(t^3\dot r(t)\bigr)'\,dt
+
\int_a^b
t^3|r(t)|^p\,dt
=
0.
\]
Integration by parts gives,
\[
\left[t^3r(t)\dot r(t)\right]_a^b
-
\int_a^b t^3|\dot r(t)|^2\,dt
+
\int_a^b t^3|r(t)|^p\,dt
=
0.
\]
Since $r(a)=r(b)=0$, the boundary term vanishes, and hence,
\[
\int_a^b t^3|\dot r(t)|^2\,dt
=
\int_a^b t^3|r(t)|^p\,dt.
\]

For $f(x)=\frac1p\Norm{x}^p$, we have,
\[
\nabla^2f(r \mathbf{e})
=
|r|^{p-2}
\left(
I+(p-2)\mathbf{e}\mathbf{e}^\top
\right).
\]
Since $h=r \mathbf{e}$, we have,
\[
\left\langle
\nabla^2f(X(t))h(t),h(t)
\right\rangle
=
(p-1)|r(t)|^p.
\]
Therefore,
\[
Q_X[h]
=
\int_a^b
t^3
\left(
|\dot r(t)|^2
-
(p-1)|r(t)|^p
\right)\,dt 
=
(2-p)
\int_a^b
t^3|r(t)|^p\,dt 
<0.
\]
Since nodal intervals occur arbitrarily far out in time, negative
second-variation directions occur on arbitrarily late intervals.
Therefore, $X$ is not eventually locally minimizing. It remains to show that $r(t)$ has a sequence of unbounded zeros. Define the Pohozaev quantity, associated
with the classical Pohozaev identity \citep{pohozaev1965eigenfunctions}, 
\begin{equation} \label{eq:pd5}
P(t)
=
t^4
\left(
\frac12|\dot r(t)|^2+\frac1p|r(t)|^p
\right)
+
t^3r(t)\dot r(t).
\end{equation}
Differentiating $P$, and using,
\[
\frac{d}{dt}|r(t)|^p
=
p|r(t)|^{p-2}r(t)\dot r(t),
\]
we obtain,
\[
\begin{aligned}
P'(t)
&=
4t^3
\left(
\frac12|\dot r(t)|^2+\frac1p|r(t)|^p
\right)
+t^4
\left(
\dot r(t)\ddot r(t)
+
|r(t)|^{p-2}r(t)\dot r(t)
\right) \\
&\quad
+3t^2r(t)\dot r(t)
+t^3|\dot r(t)|^2
+t^3r(t)\ddot r(t).
\end{aligned}
\]
The equation,
\[
\ddot r(t)+\frac3t\dot r(t)+|r(t)|^{p-2}r(t)=0,
\]
implies,
\[
t^4\dot r(t)
\left(
\ddot r(t)+|r(t)|^{p-2}r(t)
\right)
=
-3t^3|\dot r(t)|^2.
\]
It also gives,
\[
t^3r(t)\ddot r(t)
=
t^3r(t)
\left(
-\frac3t\dot r(t)
-
|r(t)|^{p-2}r(t)
\right)
=
-3t^2r(t)\dot r(t)
-
t^3|r(t)|^p.
\]
By substituting these identities into the expression for $P'(t)$, we obtain,
\[
\begin{aligned}
P'(t)
&=
4t^3
\left(
\frac12|\dot r(t)|^2+\frac1p|r(t)|^p
\right)
-3t^3|\dot r(t)|^2 \\
&\quad
+3t^2r(t)\dot r(t)
+t^3|\dot r(t)|^2
-3t^2r(t)\dot r(t)
-t^3|r(t)|^p.
\end{aligned}
\]
The terms involving $r(t)\dot r(t)$ cancel, as do the terms involving
$|\dot r(t)|^2$, since $4\left(\frac12\right)-3+1=0$. Therefore, we have,
\[
P'(t)
=
\left(\frac4p-1\right)t^3|r(t)|^p.
\]
Since $2<p<4$, the term in parentheses is positive,  so $P$ is nondecreasing. Moreover, since $P(0)=0$ and the
solution is nontrivial,
\[
P(t)>0
\qquad
\text{for every }t>0.
\]

For contradiction, suppose that $r$ has finitely many zeros.  
Since the zeros of a nontrivial solution are isolated, $r$ has a fixed sign
for all sufficiently large $t$. After possibly replacing $r$ by $-r$,
we may assume that there exists $T>0$ such that,
\begin{equation} \label{eq:pd2}
r(t)>0
\qquad
\text{for every }t\geq T.
\end{equation}
Set $F(t):=t^3\dot r(t)$.  
Then from \eqref{eq:pd1} together with  \eqref{eq:pd2}, we have
\[
F'(t) =
(t^3 \dot r(t))'
=
-t^3|r(t)|^{p-2}r(t) = -t^3 r(t)^{p-1} <0,
\qquad t\geq T,
\]
implying that  $F$ is strictly decreasing on $[T,\infty)$.

We first show that $F(t)<0$ for all sufficiently large $t$.
If such  were not the case, since $F$ is decreasing, we would have $F(t)\geq0$ for every sufficiently
large $t$. It would follow  from the definition of $F$ that $\dot r(t)\geq0$, and hence,
\[
r(t)\geq r(T)>0.
\]
Consequently,
\[
F'(t)
=
-t^3r(t)^{p-1}
\leq
-r(T)^{p-1}t^3,
\]
which forces $F(t)\to-\infty$, a contradiction. 
We conclude that, after increasing $T$ if necessary,
\[
F(t)<0
\qquad
\text{for every }t\geq T.
\]
It follows that,
\[
\dot r(t)=\frac{F(t)}{t^3}<0,
\]
so $r$ is decreasing on $[T,\infty)$. Since $r$ is positive and
decreasing, it has a limit,
\[
\ell:=\lim_{t\to\infty}r(t)\geq0.
\]
We must have $\ell=0$. Indeed, if $\ell>0$, then for all sufficiently
large $t$, we would have 
\begin{align*}
F'(t)
=
-t^3r(t)^{p-1}
\leq
-\frac{\ell^{p-1}}{2}t^3.
\end{align*}
 Let $T>0$ be such that the above inequality holds for
every $t\ge T$, and set $c_0:=\ell^{p-1}/2>0$.
Integrating from $T$ to $t$ gives,
\[
\begin{aligned}
F(t)
&=F(T)+\int_T^t F'(s)\,ds\\
&\le F(T)-c_0\int_T^t s^3\,ds\\
&=F(T)-\frac{c_0}{4}(t^4-T^4)\\
&=-\frac{c_0}{4}t^4+C,
\end{aligned}
\]
where $C:=F(T)+c_0T^4/4$ is a fixed constant.
For all sufficiently large $t$, we have $C\le c_0t^4/8$,
and therefore,
\[
F(t)\le -\frac{c_0}{8}t^4=-ct^4,
\qquad c:=\frac{c_0}{8}>0.
\]
Hence,
\[
\dot r(t)=\frac{F(t)}{t^3}\leq-ct.
\]
This would force $r$ to become negative in finite time, contradicting
$r(t)>0$. Thus $r(t) \to 0$. We now derive quantitative decay estimates. Since $F<0$ and is
decreasing, the function $-F$ is positive and increasing. For
$t\geq2T$, integrating $F'$ over $[t/2,t]$ gives,
\[
-F(t)
=
-F(t/2)
+
\int_{t/2}^{t}s^3r(s)^{p-1}\,ds.
\]
Since $-F(t/2)\geq0$ and $r$ is decreasing,
\[
r(s)\geq r(t),
\qquad
s\in[t/2,t].
\]
It follows that,
\begin{equation} \label{eq:pd3}
-F(t)
\geq
r(t)^{p-1}\int_{t/2}^{t}s^3\,ds 
=
\frac{15}{64}t^4r(t)^{p-1}.
\end{equation}
For the reverse estimate, use,
\[
-\dot r(s)=\frac{-F(s)}{s^3}.
\]
Since $-F$ is increasing, for $s\in[t,2t]$ we have,
\[
-F(s)\geq-F(t).
\]
Hence,
\[
\begin{aligned}
r(t)-r(2t)
&=
\int_t^{2t}-\dot r(s)\,ds \\
&=
\int_t^{2t}\frac{-F(s)}{s^3}\,ds \\
&\geq
-F(t)\int_t^{2t}s^{-3}\,ds \\
&=
\frac{3}{8t^2}\bigl(-F(t)\bigr).
\end{aligned}
\]
Since $r(2t)\geq0$, we also have,
\[
r(t)-r(2t)\leq r(t).
\]
It follows that
\begin{equation} \label{eq:pd4}
-F(t)
\leq
\frac83t^2r(t).
\end{equation}

Combining the estimates \eqref{eq:pd3} and \eqref{eq:pd4}, we obtain
\[
\frac{15}{64}t^4r(t)^{p-1}
\leq
-F(t)
\leq
\frac83t^2r(t).
\]
Since $r(t)>0$, dividing by $t^2r(t)$ gives,
\[
r(t)^{p-2}
\leq
C t^{-2}
\]
for some constant $C>0$. Therefore,
\[
r(t)
=
O\left(t^{-k}\right),
\qquad
k:=\frac{2}{p-2}.
\]
Since $2<p<4$, we have $k>1$. Moreover, from the definition of $F$ and \eqref{eq:pd4}, we have,
\[
|\dot r(t)|
=
\frac{-F(t)}{t^3}
\leq
C\frac{r(t)}{t},
\]
for some $C>0$, and hence $\dot r(t)
=
O\left(t^{-k-1}\right)$. We now substitute these estimates into the Pohozaev quantity. 
We have,
\[
t^4|\dot r(t)|^2
=
O\left(t^{2-2k}\right), \quad
t^3|r(t)\dot r(t)|
=
O\left(t^{2-2k}\right), \quad
t^4|r(t)|^p
=
O\left(t^{4-kp}\right).
\]
Since $k=\frac{2}{p-2}$ we have,
\[
4-kp=2-2k.
\]
Consequently, by substituting into \eqref{eq:pd5}, we obtain,
\[
|P(t)|
=
O\left(t^{2-2k}\right).
\]
Since $k>1$, the exponent on the right-hand side is negative, so,
\[
P(t)\to0.
\]
This contradicts the fact that $P(t)>0$ and $P$ is nondecreasing.
Hence $r$ must have infinitely many zeros. Note that the zeros of $r$ cannot have a finite accumulation point. If
$t_j\to t_\ast>0$ and $r(t_j)=0$, then continuity gives
$r(t_\ast)=0$, and by Rolle's theorem there is  a sequence $s_j\to t_\ast$ with $\dot r(s_j)=0$. Thus $\dot r(t_\ast)=0$, and uniqueness for the ODE
would imply that $r\equiv0$, contradicting $r(0)=\Norm{x_0}>0$.
Therefore, the set of zeros is unbounded.

\paragraph{Case $p=4$.}
In this case the regular solution is explicit,
\[
X(t)
=
\frac{x_0}
{1+\Norm{x_0}^2t^2/8}.
\]
Therefore $\Norm{X(t)}=O(t^{-2})$. Since,
\[
\nabla^2 f(X(t))
=
\Norm{X(t)}^2
\left(
I+2\mathbf{e}\mathbf{e}^\top
\right),
\]
we obtain, $t^2\nabla^2f(X(t)) \to 0$. Hence there exists $T_0>0$ and $\delta>0$ such that,
\[
t^2\nabla^2 f(X(t))
\preceq
(1-\delta)I,
\qquad
t\ge T_0.
\]
By the logarithmic-time representation of the second variation, every
finite interval,
\[
T_0<a<b<\infty,
\]
has strictly positive second variation for every nonzero fixed-endpoint
variation. Consequently, every such restriction $X|_{[a,b]}$ is a
strict local minimizer on $[a,b]$.

\paragraph{Case $p>4$.}
For $p>4$, set,
\[
k=\frac{2}{p-2}\in(0,1),
\qquad
s=\log t,
\qquad
z(s)=t^k r(t).
\]
Equivalently, $r(t)=t^{-k}z(\log t)$. Next,
the chain rule gives,
\[
\dot r(t)
=
t^{-k-1}\bigl(z_s-kz\bigr),
\]
and,
\[
\ddot r(t)
=
t^{-k-2}
\left(
z_{ss}-(2k+1)z_s+k(k+1)z
\right).
\]
Consequently,
\[
\ddot r(t)+\frac{3}{t}\dot r(t)
=
t^{-k-2}
\left(
z_{ss}+2(1-k)z_s-k(2-k)z
\right).
\]
Moreover, the identity $k(p-2)=2$ implies,
\[
|r(t)|^{p-2}r(t)
=
t^{-k(p-1)}|z|^{p-2}z
=
t^{-k-2}|z|^{p-2}z.
\]
Substituting these expressions into the equation,
\[
\ddot r+\frac{3}{t}\dot r+|r|^{p-2}r=0,
\]
and multiplying by $t^{k+2}$ yields,
\[
z_{ss}
+
2(1-k)z_s
-
k(2-k)z
+
|z|^{p-2}z
=
0.
\]
The associated energy $\mathcal E_z$ is defined by,
\[
\mathcal E_z(s)
=
\frac12|z_s|^2
-
\frac{k(2-k)}{2}|z|^2
+
\frac1p|z|^p.
\]
Differentiating with respect to $s$ gives,
\[
\begin{aligned}
\mathcal E_z'(s)
&=
z_s z_{ss}
-
k(2-k)z z_s
+
|z|^{p-2}z z_s\\
&=
z_s\left(
z_{ss}
-
k(2-k)z
+
|z|^{p-2}z
\right).
\end{aligned}
\]
The equation for $z$ implies,
\[
z_{ss}-k(2-k)z+|z|^{p-2}z
=
-2(1-k)z_s.
\]
Therefore,
\[
\mathcal E_z'(s)
=
-2(1-k)|z_s|^2
\le0,
\]
where the inequality follows from $0<k<1$. Write $r_0:=r(0)\ne0$. The scalar equation and regularity
at $t=0$ give,
\[
\dot r(t)
=
-\frac{1}{t^3}
\int_0^t \tau^3|r(\tau)|^{p-2}r(\tau)\,d\tau
=
O(t).
\]
Consequently,
\[
r(t)=r_0+O(t^2)
\qquad\text{as }t\to0.
\]
Since $z(s)=e^{ks}r(e^s)$, it follows that,
\[
z(s)
=
r_0e^{ks}+O(e^{(k+2)s}),
\]
and,
\[
z_s(s) =
e^{ks}\left(kr(e^s)+e^s\dot r(e^s)\right) =
kr_0e^{ks}+O(e^{(k+2)s}) \qquad\text{as }s\to-\infty.
\]
In particular, $z(s),z_s(s)\to0$, and hence, $\lim_{s\to-\infty}\mathcal E_z(s)=0$. Integrating $\mathcal E_z'(s)=-2(1-k)|z_s(s)|^2$ from $-\infty$ to $s$ therefore gives,
\[
\mathcal E_z(s)
=
-2(1-k)\int_{-\infty}^s|z_s(\sigma)|^2\,d\sigma
<0.
\]
The integral converges by the above asymptotics. The
inequality is strict because $r_0\ne0$ implies that $z_s$ is
nonzero for all sufficiently negative $s$. Set $\alpha:=k(2-k)>0$. Since,
\[
\mathcal E_z(s)
=
\frac12|z_s(s)|^2
-
\frac{\alpha}{2}|z(s)|^2
+
\frac1p|z(s)|^p
<0,
\]
we have $z(s)\ne0$ and,
\[
\frac1p|z(s)|^p
<
\frac{\alpha}{2}|z(s)|^2.
\]
Thus,
\[
0<|z(s)|
<
\left(\frac{p\alpha}{2}\right)^{1/(p-2)},
\]
which gives a uniform upper bound. To obtain a lower bound, fix
$s_0\in\R$ and set $\eta:=-\mathcal E_z(s_0)>0$.
Monotonicity of the energy gives, for every $s\ge s_0$,
\[
-\frac{\alpha}{2}|z(s)|^2
\le
\mathcal E_z(s)
\le
\mathcal E_z(s_0)
=
-\eta.
\]
Consequently,
\[
m:=\left(\frac{2\eta}{\alpha}\right)^{1/2}
\le
|z(s)|
<
M:=\left(\frac{p\alpha}{2}\right)^{1/(p-2)}
\qquad\text{for every }s\ge s_0.
\]
Hence $z$ remains bounded and bounded away from zero on $[s_0,\infty)$. Set,
\[
A:=k(2-k)>0,
\qquad
\beta:=2(1-k)>0,
\qquad
V(y):=-\frac{A}{2}y^2+\frac1p|y|^p.
\]
Then,
\[
z_{ss}+\beta z_s-Az+|z|^{p-2}z=0,
\]
and the energy satisfies,
\[
\mathcal E_z(s)
=
\frac12|z_s(s)|^2+V(z(s)),
\qquad
\mathcal E_z'(s)
=
-\beta|z_s(s)|^2.
\]

The energy is nonincreasing and bounded below by
$\min_{y\in\R}V(y)$. Hence there exists
$\mathcal E_\infty\in\R$ such that,
\[
\mathcal E_z(s)\longrightarrow\mathcal E_\infty
\qquad\text{as }s\to\infty.
\]
Integrating the dissipation identity gives,
\[
\beta\int_{s_0}^{\infty}|z_s(s)|^2\,ds
=
\mathcal E_z(s_0)-\mathcal E_\infty
<\infty.
\]
In particular, $z_s\in L^2((s_0,\infty))$. Furthermore,
\[
\frac12|z_s(s)|^2
=
\mathcal E_z(s)-V(z(s))
\le
\mathcal E_z(s_0)-\min_{y\in\R}V(y),
\]
so $z_s$ is bounded. The ODE then shows that,
\[
z_{ss}=-\beta z_s+Az-z^{p-1},
\]
is bounded as well. Consequently, $z_s$ is uniformly
continuous on $[s_0,\infty)$. Next, note that, 
\[
z_s(s)\longrightarrow0.
\]
If not, there exist $\varepsilon>0$
and a sequence $s_j\to\infty$ such that
$|z_s(s_j)|\ge\varepsilon$. Uniform continuity would
give $\delta>0$, independent of $j$, such that,
\[
|z_s(s)|\ge\frac{\varepsilon}{2}
\qquad\text{for }s\in[s_j,s_j+\delta].
\]
Passing to a subsequence, we may assume that these
intervals are disjoint. This contradicts
$\int_{s_0}^{\infty}|z_s|^2\,ds<\infty$. To identify the limit of $z$, let $L$ be any accumulation
point as $s\to\infty$, and choose $s_j\to\infty$ such that
$z(s_j)\to L$. Since $m\le z(s)\le M$, we have $L>0$.
Moreover,
\[
\sup_{0\le\tau\le1}|z(s_j+\tau)-L|
\le
|z(s_j)-L|
+
\sup_{\sigma\ge s_j}|z_s(\sigma)|
\longrightarrow0.
\]
Integrating the ODE over $[s_j,s_j+1]$ yields,
\[
\begin{aligned}
&z_s(s_j+1)-z_s(s_j)
+\beta\bigl(z(s_j+1)-z(s_j)\bigr)\\
&\qquad=
\int_0^1
\left(
Az(s_j+\tau)-z(s_j+\tau)^{p-1}
\right)\,d\tau.
\end{aligned}
\]
Passing to the limit gives,
\[
AL-L^{p-1}=0.
\]
Since $L>0$, necessarily,
\[
L=A^{1/(p-2)}
=
\bigl[k(2-k)\bigr]^{1/(p-2)}.
\]
Thus every accumulation point of $z$ has the same value.
Since $z$ is bounded, it follows that,
\[
z(s)\longrightarrow
\bigl[k(2-k)\bigr]^{1/(p-2)},
\qquad
z_s(s)\longrightarrow0.
\]
Finally, recalling that $r(t)=t^{-k}z(\log t)$, we obtain,
\[
r(t)\sim
\bigl[k(2-k)\bigr]^{1/(p-2)}t^{-k}
\qquad\text{as }t\to\infty.
\]
It follows that,
\[
t^2|r(t)|^{p-2}
\to
k(2-k)
=
\frac{4(p-3)}{(p-2)^2}.
\]
Therefore,
\[
t^2\nabla^2 f(X(t))
\to
K_\infty
:=
\frac{4(p-3)}{(p-2)^2}
\left(
I+(p-2)\mathbf{e}\mathbf{e}^\top
\right).
\]
The eigenvalue of $K_\infty$ in the direction $\mathbf{e}$ is,
\[
\lambda_{\parallel}(K_\infty)
=
(p-1)\frac{4(p-3)}{(p-2)^2}
=
4-\frac{4}{(p-2)^2}.
\]
Since $p>4$, we have,
\[
\lambda_{\parallel}(K_\infty)>1.
\]

Define $\delta
:=
\frac12\left(
\lambda_{\parallel}(K_\infty)-1
\right)>0$. For all sufficiently large $s$,
\[
\mathbf{e}^\top
\left(
e^{2s}\nabla^2f(X(e^s))
\right)\mathbf{e}
\ge
1+\delta.
\]

Choosing a variation of the form, $u(s)=\phi(s)\mathbf{e}$, the transformed second variation satisfies,
\[
Q_X[h]
\le
\int
\left(
|\phi_s|^2-\delta|\phi|^2
\right)\,ds.
\]
Choose $s_0$ large enough that the above
bound holds for all $s\ge s_0$. Fix
$\ell>\pi/\sqrt{\delta}$ and, on $[s_0,s_0+\ell]$, set,
\[
\phi(s)
=
\sin\left(\frac{\pi(s-s_0)}{\ell}\right).
\]
Then, $Q_X[h]<0$. Such intervals can be placed arbitrarily far out in time, so $X$ is not
eventually locally minimizing. 

By combining the four cases, we see that for $p\ge 2$, the regular Nesterov trajectory is eventually locally minimizing if and only if $p=4$.
\end{proof}

\begin{rmk}[Intuition for Theorem \ref{thm:power-characterization}]
As we explained in Corollary~\ref{cor:time-square-hessian},
$t^{2}\nabla^{2}f(X(t))$ is closely related to eventual local
minimality. With that in mind, we can better understand what
Theorem~\ref{thm:power-characterization} says.
When $p=2$, $t^{2}\nabla^{2}f(X(t))$ becomes large as $t$
becomes large, and this allows us to construct variations
with negative second variation over sufficiently long
intervals. The trajectory $r(t)$ also oscillates with
decaying amplitude.
For $2<p<4$, the trajectory oscillates and has infinitely
many zeros. Using the portions of the trajectory between
consecutive zeros as variations, we obtain negative
second variations.
For $p=4$, $t^{2}\nabla^{2}f(X(t))$ goes to zero, so we can
apply Corollary~\ref{cor:time-square-hessian}.
For $p>4$, $t^{2}\nabla^{2}f(X(t))$ has a limit whose
eigenvalue in the direction of the trajectory is greater
than $1$. Using this direction, we can construct variations
with negative second variation over sufficiently long
intervals in logarithmic time.
\end{rmk}
Theorem \ref{thm:power-characterization} provides a full
characterization of eventual local minimality for power objective
functions. We also note that the assumption $x_0\neq0$
is essential in Theorem \ref{thm:power-characterization}. When $x_0=0$, the
regular Nesterov trajectory is the trivial solution $X\equiv0$. The following proposition describes when this trajectory is eventually locally minimizing.
\begin{prop}
\label{prop:initial-zero}
When $x_0=0$, the trivial trajectory $X\equiv0$ is eventually
locally minimizing for every $p>2$, but not for $p=2$.
\end{prop}
\begin{proof}
If $p>2$, then $\nabla^2f(0)=0$, so on every finite
interval $0<a<b<\infty$ the second variation is,
\[
Q_0[h]
=
\int_a^b t^3\Norm{\dot h(t)}^2\,dt.
\]
By the Poincar\'e inequality,
\[
Q_0[h]
\ge
\frac{a^3}{1+(b-a)^2/\pi^2}
\Norm{h}_{H^1(a,b)}^2,
\]
for every $h\in H_0^1((a,b);\R^n)$.
Thus the zero trajectory is a strict $H^1$-local
minimizer on every such interval $[a,b]$, and hence is
eventually locally minimizing for every $p>2$. For $p=2$, the second variation is independent of
the underlying trajectory. Consequently, the
instability argument for the quadratic case also
applies to the zero trajectory, which is not
eventually locally minimizing.
\end{proof}
We have presented a self-contained proof based on
ODE analysis. The proof is inspired by a connection between
Nesterov trajectories and radial solutions of a class of
semilinear elliptic PDEs. We elaborate on this connection in the
next section, shedding more light on why the exponent $p=4$ is special
in our analysis.
\section{Generalized Nesterov Dynamics and Semilinear Elliptic PDEs}
\label{sec:pde-interpretation}
A natural question is why $p=4$ is the only exponent for which non-trivial Nesterov trajectories are eventually locally minimizing, regardless of the dimension of the optimization problem. One way to understand this phenomenon is through the connection between these trajectories and radial solutions of a class of semilinear elliptic PDEs known as Lane--Emden equations \citep{gidas1981global,farina2007classification,fazly2019monotonicity}. In this section, we elaborate on this connection. 

\subsection{Radial Solutions and Elliptic Stability}

Let \(\gamma:[a,b]\to\R^n\) be a \(C^2\) curve satisfying the {\it generalized Nesterov ODE},
\[
\ddot{\gamma}(t)+\frac{d-1}{t}\dot{\gamma}(t)+\nabla f(\gamma(t))=0.
\]
Define the annulus $D_{a,b}:=\left\{
\xi\in\R^d:
a<\Norm{\xi}<b
\right\}$ and the radial map,
\begin{align}
\label{eq:radial-map}
U:D_{a,b}\to\R^n, \quad
U(\xi)=\gamma(\Norm{\xi})    
\end{align}

Then we have the following lemma.

\begin{lem}
\label{lem:bridge}
The function \(U\) satisfies the semilinear elliptic system,
\[
\Delta U+\nabla f(U)=0
\qquad\text{in }D_{a,b}.
\]
\end{lem}

\begin{proof}
Let $r=\Norm{\xi}$. Since \(U(\xi)=\gamma(r)\) is radial, the Laplacian in \(\R^d\) is given componentwise by,
\[
\Delta U(\xi)
=
\ddot{\gamma}(r)+\frac{d-1}{r}\dot{\gamma}(r).
\]
Since \(\gamma\) satisfies the generalized Nesterov ODE,
\[
\ddot{\gamma}(r)
+
\frac{d-1}{r}\dot{\gamma}(r)
+
\nabla f(\gamma(r))
=
0.
\]
Using \(U(\xi)=\gamma(r)\), we obtain,
\[
\Delta U(\xi)+\nabla f(U(\xi))=0.
\]
\end{proof}

For a $C^2$ function $f:\R^n\to\R$, Lemma \ref{lem:bridge} states that radial solutions of the PDE system,
\begin{align}
\label{eq:semi-linear-PDE}
\Delta U+\nabla f(U)=0,
\end{align}

correspond to generalized Nesterov trajectories, and vice versa. This system is variational where $U:\R^d\to\R^n$. Associated with this system is the formal energy functional,
\[
\mathcal J(U)
=
\int_{\R^d}
\left(
\frac{1}{2}\Norm{\nabla U}^2
-
f(U)
\right)\,dx.
\]

Let $U$ be a solution of the elliptic system. The second variation of
$\mathcal{J}$ at $U$ in the direction
$\Phi\in C_c^\infty(\R^d;\R^n)$ is given by,
\[
Q_U(\Phi)
:=
\delta^2\mathcal J(U)[\Phi,\Phi]
=
\int_{\R^d}
\left(
\Norm{\nabla \Phi}^2
-
\left\langle
\nabla^2 f(U)\Phi,\Phi
\right\rangle
\right)\,dx.
\]

\begin{deff}[Stable solution]
\label{deff:stable-sol-elliptic}
A solution $U$ is called \emph{stable} if,
\[
Q_U(\Phi)\geq 0,
\]
for every $\Phi\in C_c^\infty(\R^d;\R^n)$.
\end{deff}

\begin{deff}[Stable outside a compact set]
\label{deff:stable-out-compact}
A solution $U$ is called \emph{stable outside a compact set} if there exists $R_0>0$ such that,
\[
Q_U(\Phi)\geq 0,
\]
for every $\Phi\in C_c^\infty(\R^d\setminus B_{R_0};\R^n)$.
\end{deff}
\subsection{Lane–Emden Solutions}
Let $f(x)=\frac{1}{p}\Norm{x}^p,p\ge2$. Then,
\[
\nabla f(x)=\Norm{x}^{p-2}x,
\]
Thus the elliptic system (\ref{eq:semi-linear-PDE}) becomes,
\[
\Delta U+\Norm{U}^{p-2}U=0,
\qquad U:\R^d\to\R^n.
\]
This is the vector-valued Lane--Emden system \citep{gidas1981global,farina2007classification, fazly2019monotonicity}. In the scalar
case, $u:\R^d\to\R$, it reads,
\[
\Delta u+|u|^{p-2}u=0.
\]
To compare our notation with the standard Lane--Emden equation,
write,
\[
-\Delta u=|u|^{q-1}u,
\]
which becomes $-\Delta u=u^q$ for positive solutions.
For our power objective, the nonlinearity is
$|u|^{p-2}u$. Matching the powers gives,
\[
q-1=p-2,
\qquad\text{that is,}\qquad
q=p-1.
\]
 \begin{rmk}[What is special about $p=4$?]
The Sobolev critical exponent in dimension $d\ge3$ is $p=2d/(d-2)$, which equals $4$ when $d=4$. At this exponent, the energy $\mathcal J$ is invariant under the rescaling,
\[
u_\lambda(x)=\lambda^{(d-2)/2}u(\lambda x),
\qquad \lambda>0.
\]
The associated scalar equation admits explicit positive solutions known as Aubin--Talenti bubbles \citep{aubin1976problemes,talenti1976best}. These bubbles have Morse index one and are therefore stable outside a compact set. When $d=4$, the bubbles centered at the origin can be identified, through the radial map \eqref{eq:radial-map}, with the Nesterov trajectories appearing in the proof of Theorem \ref{thm:power-characterization} for $p=4$. Moreover, one can use this scaling invariance and Noether's theorem to obtain the Pohozaev quantity used in that proof. This quantity is conserved when $p=4$, while its derivative has a fixed sign when $p\neq4$. These properties play a key role in our analysis of Theorem \ref{thm:power-characterization} and explain why $p=4$ is special in the standard Nesterov dynamics. When $p>4$, the Joseph--Lundgren theory \citep{JosephLundgren1973} explains when the rescaled phase-plane trajectory spirals toward its positive equilibrium. For $d=4$, the Joseph--Lundgren exponent is infinite, and this spiralling behavior occurs for every $p>4$. 
\end{rmk}

\section*{AI Statement}
GPT-6 Astra (OpenAI) drew our attention to relevant results on the stability of solutions to semilinear elliptic equations. We then studied the related PDE literature and developed and verified a self-contained proof of the main theorem. Beyond this, we used AI to polish the text, draft presentations of known results, refine proof arguments, and obtain feedback on the manuscript. The authors have checked and verified all content line by line and take full responsibility for it.

\addcontentsline{toc}{section}{References}
\bibliography{refs}

@article{JosephLundgren1973,
  author  = {Joseph, D. D. and Lundgren, T. S.},
  title   = {Quasilinear {D}irichlet problems driven by positive sources},
  journal = {Archive for Rational Mechanics and Analysis},
  volume  = {49},
  pages   = {241--269},
  year    = {1973}
}

@article{aubin1976problemes,
  title={Probl{\`e}mes isop{\'e}rim{\'e}triques et espaces de {Sobolev}},
  author={Aubin, Thierry},
  journal={Journal of differential geometry},
  volume={11},
  number={4},
  pages={573--598},
  year={1976},
  publisher={Lehigh University}
}

@article{talenti1976best,
  title={Best constant in {Sobolev} inequality},
  author={Talenti, Giorgio},
  journal={Annali di Matematica pura ed Applicata},
  volume={110},
  number={1},
  pages={353--372},
  year={1976},
  publisher={Springer-Verlag, Berlin/Heidelberg}
}

@article{gidas1981global,
  title={Global and local behavior of positive solutions of nonlinear elliptic equations},
  author={Gidas, Basilis and Spruck, Joel},
  journal={Communications on Pure and Applied Mathematics},
  volume={34},
  number={4},
  pages={525--598},
  year={1981},
  publisher={Wiley Subscription Services, Inc., A Wiley Company, New York}
}

@article{fazly2019monotonicity,
  title={Monotonicity formulas for coupled elliptic gradient systems with applications},
  author={Fazly, Mostafa and Shahgholian, Henrik},
  journal={Advances in Nonlinear Analysis},
  volume={9},
  number={1},
  pages={479--495},
  year={2020},
  publisher={De Gruyter}
}

@inproceedings{pohozaev1965eigenfunctions,
  title={Eigenfunctions of the equation {$\Delta u+ \lambda f (u)= 0$}},
  author={Pohozaev, Stanislav},
  booktitle={Soviet Math. Doklady.},
  volume={6},
  pages={1408--1411},
  year={1965}
}

@article{wright2022optimization,
  title={Optimization for data analysis},
  author={Wright, Stephen J and Recht, Benjamin},
  journal={Cambridge University Press},
  year={2022}
}

@article{farina2007classification,
  title={On the classification of solutions of the {Lane--Emden} equation on unbounded domains of {$\mathbb{R}^{N}$}},
  author={Farina, Alberto},
  journal={Journal de math{\'e}matiques pures et appliqu{\'e}es},
  volume={87},
  number={5},
  pages={537--561},
  year={2007},
  publisher={Elsevier}
}

@inproceedings{muehlebach2019dynamical,
  title={A dynamical systems perspective on {Nesterov} acceleration},
  author={Muehlebach, Michael and Jordan, Michael},
  booktitle={Proceedings of the 36th International Conference on Machine Learning},
  series={Proceedings of Machine Learning Research},
  volume={97},
  pages={4656--4662},
  year={2019},
  publisher={PMLR}
}

@article{muehlebach2021optimization,
  title={Optimization with momentum: Dynamical, control-theoretic, and symplectic perspectives},
  author={Muehlebach, Michael and Jordan, Michael I},
  journal={Journal of Machine Learning Research},
  volume={22},
  number={73},
  pages={1--50},
  year={2021}
}

@article{kusano1987oscillation,
  title={Oscillation theory of entire solutions of second order superlinear elliptic equations},
  author={Kusano, Takasi and Naito, Manabu},
  journal={Funkcial. Ekvac},
  volume={30},
  number={2-3},
  pages={269--282},
  year={1987}
}

@article{kawano1990generalized,
  title={A generalized {Pohozaev} identity and its applications},
  author={Kawano, Nichiro and Ni, Wei-Ming and Yotsutani, Shoji},
  journal={Journal of the Mathematical Society of Japan},
  volume={42},
  number={3},
  pages={541--564},
  year={1990},
  publisher={The Mathematical Society of Japan}
}

@article{nesterov1983method,
  author  = {Nesterov, Yurii},
  title   = {A method for solving the convex programming problem
             with convergence rate {$O(1/k^2)$}},
  journal = {Doklady Akademii Nauk {SSSR}},
  volume  = {269},
  number  = {3},
  pages   = {543--547},
  year    = {1983}
}

@article{wibisono2016variational,
  title={A variational perspective on accelerated methods in optimization},
  author={Wibisono, Andre and Wilson, Ashia C and Jordan, Michael I},
  journal={Proceedings of the National Academy of Sciences},
  volume={113},
  number={47},
  pages={E7351--E7358},
  year={2016},
  publisher={National Academy of Sciences}
}

@article{su2016differential,
  title={A differential equation for modeling {Nesterov}'s accelerated gradient method: Theory and insights},
  author={Su, Weijie and Boyd, Stephen and {Candès}, Emmanuel J},
  journal={Journal of Machine Learning Research},
  volume={17},
  number={153},
  pages={1--43},
  year={2016}
}

@inproceedings{jordan2018dynamical,
  title={Dynamical, symplectic and stochastic perspectives on gradient-based optimization},
  author={Jordan, Michael I},
  booktitle={Proceedings of the international congress of mathematicians: Rio de Janeiro 2018},
  pages={523--549},
  year={2018},
  organization={World Scientific}
}

@article{zhang2021rethinking,
  title={Rethinking the variational interpretation of accelerated optimization methods},
  author={Zhang, Peiyuan and Orvieto, Antonio and Daneshmand, Hadi},
  journal={Advances in Neural Information Processing Systems},
  volume={34},
  pages={14396--14406},
  year={2021}
}
\bibliographystyle{abbrvnat}

\appendix

\section{Additional Proofs}
\label{sec:proof}

% \sw{I suggest to move this section to an appendix}\peyman{standard proof are now in this section}

In this section, we provide the proofs omitted from earlier sections. These proofs are standard and are included for completeness.
\begin{proof}[Proof of Theorem \ref{thm:cont-grad-descent}]
We first observe that the objective decreases along the gradient flow. Indeed,
\[
\frac{d}{dt}f(X(t))
=
\left\langle
\nabla f(X(t)),\dot X(t)
\right\rangle.
\]
Using $\dot X(t)=-\nabla f(X(t))$, we obtain,
\[
\frac{d}{dt}f(X(t))
=
-\Norm{\nabla f(X(t))}^2
\le0.
\]
Hence \(t\mapsto f(X(t))\) is nonincreasing. Now fix \(x^\star\in\operatorname{argmin} f\), and define, $V(t)
=
\frac12\Norm{X(t)-x^\star}^2$. Differentiating gives,
\[
\dot V(t)
=
\left\langle
X(t)-x^\star,\dot X(t)
\right\rangle
=
-\left\langle
X(t)-x^\star,\nabla f(X(t))
\right\rangle.
\]
By convexity of \(f\), we have,
\begin{equation} \label{eq:so1}
\left\langle
\nabla f(X(t)),X(t)-x^\star
\right\rangle
\ge
f(X(t))-f(x^\star).
\end{equation}
Therefore,
\[
\dot V(t)
\le
-\left(f(X(t))-f(x^\star)\right).
\]
Integrating from \(0\) to \(t\), we obtain,
\[
\int_0^t
\left(
f(X(s))-f(x^\star)
\right)\,ds
\le
V(0)-V(t)
\le
V(0).
\]
Since \(f(X(s))\) is nonincreasing,
\[
f(X(s))-f(x^\star)
\ge
f(X(t))-f(x^\star),
\qquad 0\le s\le t.
\]
Consequently, we have,
\[
t\left(f(X(t))-f(x^\star)\right)
\le
\int_0^t
\left(
f(X(s))-f(x^\star)
\right)\,ds
\le
\frac12\Norm{x_0-x^\star}^2.
\]
Dividing by \(t\) gives,
\[
f(X(t))-f(x^\star)
\le
\frac{\Norm{x_0-x^\star}^2}{2t},
\]
which proves the claim.
\end{proof}
\begin{proof}[Proof of Theorem \ref{thm:nesterov}]
Fix \(x^\star\in\operatorname{argmin} f\), and consider the Lyapunov function, 
\begin{equation} \label{eq:defE}
\mathcal E(t)
=
t^2\bigl(f(X(t))-f(x^\star)\bigr)
+
2\Norm{
X(t)+\frac{t}{2}\dot X(t)-x^\star
}^2.
\end{equation}
We claim that \(\mathcal E(t)\) is nonincreasing. Let,
\[
Z(t)=X(t)+\frac{t}{2}\dot X(t)-x^\star.
\]
Then,
\[
\dot Z(t)
=
\dot X(t)+\frac{1}{2}\dot X(t)+\frac{t}{2}\ddot X(t)
=
\frac{3}{2}\dot X(t)+\frac{t}{2}\ddot X(t).
\]
Using the Nesterov equation,
\[
\ddot X(t)
=
-\frac{3}{t}\dot X(t)-\nabla f(X(t)),
\]
we obtain,
\[
\dot Z(t)
=
-\frac{t}{2}\nabla f(X(t)).
\]

Now differentiate \(\mathcal E(t)\) to obtain,
\[
\begin{aligned}
\dot{\mathcal E}(t)
&=
2t\bigl(f(X(t))-f(x^\star)\bigr)
+
t^2\left\langle \nabla f(X(t)),\dot X(t)\right\rangle
+
4\langle Z(t),\dot Z(t)\rangle \\
&=
2t\bigl(f(X(t))-f(x^\star)\bigr)
+
t^2\left\langle \nabla f(X(t)),\dot X(t)\right\rangle
-
2t\left\langle Z(t),\nabla f(X(t))\right\rangle.
\end{aligned}
\]
By substituting the definition of \(Z(t)\), we find that,
\[
\begin{aligned}
\dot{\mathcal E}(t)
&=
2t\bigl(f(X(t))-f(x^\star)\bigr)
+
t^2\left\langle \nabla f(X(t)),\dot X(t)\right\rangle -2t
\left\langle
X(t)+\frac{t}{2}\dot X(t)-x^\star,
\nabla f(X(t))
\right\rangle.
\end{aligned}
\]
The terms involving \(\dot X(t)\) cancel, and hence,
\[
\dot{\mathcal E}(t)
=
2t\left(
f(X(t))-f(x^\star)
-
\left\langle
\nabla f(X(t)),X(t)-x^\star
\right\rangle
\right).
\]
By using convexity of \(f\) and substituting from \eqref{eq:so1}, we have that
$\dot{\mathcal E}(t)\le0$. Thus \(\mathcal E(t)\) is nonincreasing. Finally, using the regular initial conditions, $\mathcal E(0)=2\Norm{x_0-x^\star}^2$. Since the second term in \(\mathcal E(t)\) is nonnegative, we obtain
\[
t^2\bigl(f(X(t))-f(x^\star)\bigr)
\le
\mathcal E(t)
\le
\mathcal E(0)
=
2\Norm{x_0-x^\star}^2.
\]
Dividing by \(t^2\) yields the result \eqref{eq:so2}.
\end{proof}
\begin{proof}[Proof of Lemma \ref{lem:reg-nesterov-sol}]
Multiplying the differential equation by $t^3$ gives,
\[
\bigl(t^3\dot X(t)\bigr)'
=
-t^3\nabla f(X(t)).
\]
For a regular solution satisfying $\dot X(0)=0$, integration from $0$ to
$t$ yields,
\[
t^3\dot X(t)
=
-\int_0^t s^3\nabla f(X(s))\,ds,
\]
and therefore,
\[
\dot X(t)
=
-\frac1{t^3}
\int_0^t s^3\nabla f(X(s))\,ds.
\]
Integrating once more from $0$ to $t$ and using $X(0)=x_0$, we obtain,
\[
X(t)
=
x_0-
\int_0^t
\tau^{-3}
\left(
\int_0^\tau
s^3\nabla f(X(s))\,ds
\right)d\tau.
\]

We now prove local existence and uniqueness from this integral equation.
Fix $R>0$, and let,
\[
B_R(x_0)
=
\left\{
x\in\R^n:\Norm{x-x_0}\leq R
\right\}.
\]
Since $\nabla f$ is locally Lipschitz, there are constants $M,L>0$
such that,
\[
\Norm{\nabla f(x)}\leq M,
\]
and,
\[
\Norm{\nabla f(x)-\nabla f(y)}
\leq
L\Norm{x-y},
\]
for all $x,y\in B_R(x_0)$. Consider the closed subset,
\[
\mathcal B
=
\left\{
X\in C([0,T];\R^n):
\sup_{0\leq t\leq T}\Norm{X(t)-x_0}\leq R
\right\},
\]
and define,
\[
(\mathcal TX)(t)
=
x_0-
\int_0^t
\tau^{-3}
\left(
\int_0^\tau
s^3\nabla f(X(s))\,ds
\right)d\tau.
\]
For $X\in\mathcal B$,
\[
\begin{aligned}
\Norm{(\mathcal TX)(t)-x_0}
&\leq
M
\int_0^t
\tau^{-3}
\left(
\int_0^\tau s^3\,ds
\right)d\tau \\
&=
\frac{M}{4}
\int_0^t \tau\,d\tau \\
&=
\frac{M}{8}t^2.
\end{aligned}
\]
Hence $\mathcal T$ maps $\mathcal B$ into itself when $\frac{M}{8}T^2\leq R$. Similarly, for $X,Y\in\mathcal B$,
\[
\begin{aligned}
\Norm{(\mathcal TX)(t)-(\mathcal TY)(t)}
&\leq
L\Norm{X-Y}_{C^0}
\int_0^t
\tau^{-3}
\left(
\int_0^\tau s^3\,ds
\right)d\tau \\
&=
\frac{L}{8}t^2\Norm{X-Y}_{C^0}.
\end{aligned}
\]
Therefore,
\[
\Norm{\mathcal TX-\mathcal TY}_{C^0}
\leq
\frac{L}{8}T^2
\Norm{X-Y}_{C^0}.
\]
Choosing $T$ even smaller, if necessary, so that, $\frac{L}{8}T^2<1$, the map $\mathcal T$ is a contraction. By the Banach fixed-point theorem there is a unique $X\in\mathcal B$ satisfying the integral
equation. Next, we verify the behavior at $t=0$. Since $X(t)\to x_0$ and
$\nabla f$ is continuous,
\[
\nabla f(X(s))
=
\nabla f(x_0)+o(1)
\qquad
\text{as }s\to0.
\]
Thus,
\[
\begin{aligned}
\dot X(t)
&=
-\frac1{t^3}
\int_0^t
s^3
\left(
\nabla f(x_0)+o(1)
\right)ds\\
&=
-\frac{t}{4}\nabla f(x_0)+o(t).
\end{aligned}
\]
In particular,
\[
\dot X(t)\to0,
\]
so $\dot X(0)=0$. Integrating this expansion gives,
\[
X(t)
=
x_0-\frac{t^2}{8}\nabla f(x_0)+o(t^2).
\]

Finally, differentiating the integral equation for $t>0$ shows that
the fixed point is a classical solution of the Nesterov equation on
$(0,T]$. This proves existence and uniqueness of the regular local
solution.
\end{proof}

\begin{proof}[Proof of Lemma \ref{lem:global-existence}]
Lemma~2.2 gives a unique regular local solution.
Let $[0,T_{\max})$ be its maximal interval of existence.
Define the mechanical energy,
\[
E(t):=\frac12\|\dot X(t)\|^2+f(X(t)).
\]
For $t>0$, the equation gives,
\[
E'(t)
=
\left\langle
\dot X(t),\ddot X(t)+\nabla f(X(t))
\right\rangle
=
-\frac{3}{t}\|\dot X(t)\|^2
\le 0.
\]
The regular initial conditions therefore imply,
\[
E(t)\le E(0)=f(x_0).
\]

Set $m:=\inf_{\mathbb R^n}f>-\infty$.
Then,
\[
\frac12\|\dot X(t)\|^2
\le f(x_0)-f(X(t))
\le f(x_0)-m.
\]
Consequently, with
$C:=\sqrt{2(f(x_0)-m)}$, we have,
\[
\|\dot X(t)\|\le C,
\qquad
\|X(t)-x_0\|\le Ct,
\qquad 0\le t<T_{\max}.
\]

If $T_{\max}<\infty$, both $X(t)$
and $\dot X(t)$ remain bounded as $t\uparrow T_{\max}$. Now a standard continuation argument extends the
solution beyond $T_{\max}$, a contradiction. Therefore $T_{\max}=\infty$.
\end{proof}
Before proving Lemma \ref{lem:Coercive-second-variation-local-minimality}, we prove the following simple lemma which says that strict positivity of second variation on a finite interval implies coercivity. This lemma implies that Lemma \ref{lem:Coercive-second-variation-local-minimality} also holds when second variation operator is only assumed to be strictly positive.
\begin{lem}
\label{lem:strict-positivity-implies-coercivity}
Let $0<a<b<\infty$, and let
$A\in L^\infty((a,b);\R^{n\times n})$. Define,
\[
Q(h)
:=
\int_a^b t^3
\left(
\Norm{\dot h(t)}^2
-
\langle A(t)h(t),h(t)\rangle
\right)\,dt,
\qquad
h\in H_0^1((a,b);\R^n).
\]
Suppose that $Q(h)>0$ for every nonzero
$h\in H_0^1((a,b);\R^n)$. Then there exists $c>0$ such that,
\[
Q(h)\ge c\Norm{h}_{H^1(a,b)}^2
\qquad
\text{for every }h\in H_0^1((a,b);\R^n).
\]
\end{lem}

\begin{proof}
Suppose, to the contrary, that no such $c$ exists.
Then there exists a sequence
$h_j\in H_0^1((a,b);\R^n)$ satisfying,
\[
\Norm{h_j}_{H^1(a,b)}=1,
\qquad
Q(h_j)\longrightarrow 0.
\]
By passing to a subsequence we have,
\[
h_j\rightharpoonup h
\quad\text{in }H_0^1((a,b);\R^n),
\qquad
h_j\longrightarrow h
\quad\text{in }L^2((a,b);\R^n),
\]
where the first convergence follows from weak compactness in $H^{1}_{0}$ and second one follows from compact embedding of $H^{1}_{0}$ into $L^{2}$.
Since $A$ is bounded, the $\langle A(t)h(t),h(t)\rangle$ term converges. Next, weak lower semicontinuity of the $L^2$
norm implies,
\[
Q(h)\le \liminf_{j\to\infty}Q(h_j)=0.
\]
Strict positivity therefore forces $h=0$.
Consequently,
\[
a^{3}\int_a^b \Norm{\dot h_j(t)}^2\,dt
\le 
\int_a^b t^3\Norm{\dot h_j(t)}^2\,dt
=
Q(h_j)
+
\int_a^b t^3
\langle A(t)h_j(t),h_j(t)\rangle\,dt
\longrightarrow 0.
\]
Hence 
$\dot h_j\to0$ in $L^2$. Together with $h_j\to0$ in $L^2$,
this gives $\Norm{h_j}_{H^1(a,b)}\to0$, contradicting
the normalization.
\end{proof}
\begin{proof}[Proof of Lemma \ref{lem:Coercive-second-variation-local-minimality}]
Since $\gamma$ is a critical point, its first variation vanishes,
\[
\delta\mathcal{L}(\gamma)[h]=0
\qquad
\text{for every }h\in H_0^1((a,b);\R^n).
\]

Fix $\rho>0$ and consider the compact neighborhood,
\[
K_\rho
:=
\left\{
x\in\R^n:
\operatorname{dist}(x,\gamma([a,b]))\le\rho
\right\}.
\]
Since $f\in C^2(\R^n)$, its Hessian is uniformly continuous on
$K_\rho$. Define,
\[
\omega(r)
:=
\sup_{\substack{x,y\in K_\rho\\ \Norm{x-y}\le r}}
\Norm{\nabla^2f(x)-\nabla^2f(y)}_{\mathrm{op}}.
\]
Then $\omega$ is nondecreasing and satisfies
$\omega(r)\to0$ as $r\to0$. Consider $h\in H_0^1((a,b);\R^n)$ with
$\Norm{h}_{L^\infty}\le\rho$.
For each fixed $t\in[a,b]$, define,
\[
g_t(\theta)
:=
f(\gamma(t)+\theta h(t)),
\qquad \theta\in[0,1].
\]
The chain rule gives,
\[
g_t'(0)
=
\left\langle\nabla f(\gamma(t)),h(t)\right\rangle,
\]
and,
\[
g_t''(\theta)
=
\left\langle
\nabla^2f(\gamma(t)+\theta h(t))h(t),h(t)
\right\rangle.
\]
Applying the one-variable Taylor formula,
\[
g_t(1)
=
g_t(0)+g_t'(0)
+
\int_0^1(1-\theta)g_t''(\theta)\,d\theta,
\]
and adding and subtracting $g_t''(0)$ inside the integral yields,
\[
\begin{aligned}
f(\gamma(t)+h(t))
={}&
f(\gamma(t))
+
\left\langle\nabla f(\gamma(t)),h(t)\right\rangle\\
&+
\frac12
\left\langle
\nabla^2f(\gamma(t))h(t),h(t)
\right\rangle
+
R(t,h),
\end{aligned}
\]
where,
\[
R(t,h)
:=
\int_0^1(1-\theta)
\left\langle
\left[
\nabla^2f(\gamma(t)+\theta h(t))
-
\nabla^2f(\gamma(t))
\right]h(t),h(t)
\right\rangle\,d\theta.
\]
All points $\gamma(t)+\theta h(t)$ belong to $K_\rho$.
Consequently,
\[
\begin{aligned}
|R(t,h)|
&\le
\int_0^1(1-\theta)
\omega\bigl(\theta\Norm{h(t)}\bigr)
\Norm{h(t)}^2\,d\theta\\
&\le
\frac12\omega(\Norm{h}_{L^\infty})
\Norm{h(t)}^2.
\end{aligned}
\]

The kinetic term expands as follows for almost every $t\in(a,b)$,
\[
\frac12\Norm{\dot\gamma(t)+\dot h(t)}^2
=
\frac12\Norm{\dot\gamma(t)}^2
+
\left\langle\dot\gamma(t),\dot h(t)\right\rangle
+
\frac12\Norm{\dot h(t)}^2.
\]
Combining this identity with the Taylor expansion of the potential
term gives,
\[
\begin{aligned}
\mathcal{L}(\gamma+h)-\mathcal{L}(\gamma)
={}&
\int_a^b t^3
\left[
\frac12\left(
\Norm{\dot\gamma+\dot h}^2-\Norm{\dot\gamma}^2
\right)
-
\left(f(\gamma+h)-f(\gamma)\right)
\right]dt\\
={}&
\int_a^b t^3
\left(
\left\langle\dot\gamma,\dot h\right\rangle
-
\left\langle\nabla f(\gamma),h\right\rangle
\right)dt\\
&+
\frac12\int_a^b t^3
\left(
\Norm{\dot h}^2
-
\left\langle\nabla^2f(\gamma)h,h\right\rangle
\right)dt
-
\int_a^b t^3R(t,h)\,dt.
\end{aligned}
\]
Here the first integral contains the terms linear in $h$, and
the second contains the quadratic terms. By the formulas for
the first and second variations,
\[
\delta\mathcal{L}(\gamma)[h]
=
\int_a^b t^3
\left(
\left\langle\dot\gamma,\dot h\right\rangle
-
\left\langle\nabla f(\gamma),h\right\rangle
\right)dt,
\]
and,
\[
\delta^2\mathcal{L}(\gamma)[h,h]
=
\int_a^b t^3
\left(
\Norm{\dot h}^2
-
\left\langle\nabla^2f(\gamma)h,h\right\rangle
\right)dt.
\]
Therefore,
\[
\mathcal{L}(\gamma+h)-\mathcal{L}(\gamma)
=
\delta\mathcal{L}(\gamma)[h]
+
\frac12\delta^2\mathcal{L}(\gamma)[h,h]
-
\int_a^b t^3R(t,h)\,dt.
\]
Since $\gamma$ is a critical point and $h$ preserves the endpoints,
$\delta\mathcal{L}(\gamma)[h]=0$. Hence,
\[
\mathcal{L}(\gamma+h)-\mathcal{L}(\gamma)
=
\frac12\delta^2\mathcal{L}(\gamma)[h,h]
-
\int_a^b t^3R(t,h)\,dt.
\]
By coercivity,
\[
\frac12\delta^2\mathcal{L}(\gamma)[h,h]
\ge
\frac{c}{2}\Norm{h}_{H^1(a,b)}^2.
\]

Since the time interval is one-dimensional, the Sobolev embedding,
\[
H^1((a,b);\R^n)\hookrightarrow C^0([a,b];\R^n),
\]
implies that there exists $C_1>0$ such that,
\[
\Norm{h}_{L^\infty}
\le
C_1\Norm{h}_{H^1(a,b)}.
\]
Using the remainder estimate and $t^3\le b^3$, we obtain,
\[
\begin{aligned}
\left|\int_a^b t^3R(t,h)\,dt\right|
&\le
\frac{b^3}{2}
\omega(\Norm{h}_{L^\infty})
\Norm{h}_{L^2(a,b)}^2\\
&\le
\frac{b^3}{2}
\omega\left(C_1\Norm{h}_{H^1(a,b)}\right)
\Norm{h}_{H^1(a,b)}^2.
\end{aligned}
\]
Therefore,
\[
\mathcal{L}(\gamma+h)-\mathcal{L}(\gamma)
\ge
\left[
\frac{c}{2}
-
\frac{b^3}{2}
\omega\left(C_1\Norm{h}_{H^1(a,b)}\right)
\right]
\Norm{h}_{H^1(a,b)}^2.
\]

Since $\omega(r)\to0$ as $r\to0$, we can choose
$\varepsilon>0$ sufficiently small so that,
\[
C_1\varepsilon\le\rho,
\qquad
\frac{b^3}{2}\omega(C_1\varepsilon)\le\frac{c}{4}.
\]
Hence every nonzero $h\in H_0^1((a,b);\R^n)$ satisfying
$\Norm{h}_{H^1(a,b)}<\varepsilon$ also satisfies
$\Norm{h}_{L^\infty}\le\rho$, and,
\[
\mathcal{L}(\gamma+h)-\mathcal{L}(\gamma)
\ge
\frac{c}{4}\Norm{h}_{H^1(a,b)}^2
>0.
\]
Thus $\gamma$ is a strict $H^1$-local minimizer on $[a,b]$.
\end{proof}

\end{document}